\documentclass[pdflatex,sn-mathphys-num]{sn-jnl}% Math and Physical Sciences Numbered Reference Style
\usepackage{graphicx}%
\usepackage{multirow}%
\usepackage{amsmath,amssymb,amsfonts}%
\usepackage{amsthm}%
\usepackage{mathrsfs}%
\usepackage[title]{appendix}%
\usepackage{xcolor}%
\usepackage{textcomp}%
\usepackage{manyfoot}%
\usepackage{booktabs}%
\usepackage{algorithm}%
\usepackage{algorithmicx}%
\usepackage{algpseudocode}%
\usepackage{listings}%
\theoremstyle{thmstyleone}%
\theoremstyle{thmstyletwo}%

\theoremstyle{thmstylethree}%

\usepackage{xcolor}
\usepackage{bm}
\usepackage{graphics}
\definecolor{refgreen}{rgb}{0,0.5,0}
\usepackage{tikz}
\usetikzlibrary{arrows}
\usepackage{adjustbox}

\usepackage{leftidx}

\usepackage{amsthm}
\newtheorem{thm}{Theorem}[section]
\newtheorem{lem}[thm]{Lemma}
\newtheorem{prop}[thm]{Proposition}
\newtheorem{rem}[thm]{Remark}

\newcommand\bfd{{\mathbf d}}
\newcommand\bfe{{\mathbf e}}
\newcommand\bff{{\mathbf f}}
\newcommand\bfg{{\mathbf g}}

\newcommand\bfn{{\mathbf n}}
\newcommand\bfr{{\mathbf r}}

\newcommand\bfu{{\mathbf u}}
\newcommand\bfv{{\mathbf v}}
\newcommand\bfw{{\mathbf w}}
\newcommand\bfx{{\mathbf x}}
\newcommand\bfy{{\mathbf y}}
\newcommand\bfz{{\mathbf z}}
\newcommand\bfA{{\mathbf A}}

\newcommand\bfE{{\mathbf E}}
\newcommand\bfH{{\mathbf H}}
\newcommand\bfF{{\mathbf F}}

\newcommand\bfK{{\mathbf K}}
\newcommand\bfR{{\mathbf R}}
\newcommand\bfM{{\mathbf M}}

\newcommand\bbM{{\mathbf M}}
\newcommand\bbA{{\mathbf A}}

\newcommand\bfV{{\mathbf V}}

\newcommand\bfvartheta{{\boldsymbol \vartheta}}

\newcommand\calF{{\cal F}}
\newcommand\calO{{\mathcal O}}

\newcommand{\D}{{\textnormal d}}
\newcommand{\R}{\mathbb{R}}

\DeclareMathOperator{\Real}{Re}

\newcommand{\Ga}{\varGamma}

\renewcommand{\nu}{\textnormal{n}}
\newcommand{\matdev}{\partial^{\bullet}}

\newcommand{\TbfM}{\leftidx{^\top}{\mathbf{M}}{}}

\newcommand{\TbfK}{\leftidx{^\top}{\bfK}{}}
\newcommand{\bfde}{\mathbf{\dot e}}
\newcommand{\bfdu}{\mathbf{\dot u}}
\newcommand{\bfdx}{\mathbf{\dot x}}
\newcommand{\bfdd}{\mathbf{\dot d}}
\newcommand{\bftx}{\mathbf{\widetilde x}}
\newcommand{\bftu}{\mathbf{\widetilde u}}
\newcommand{\bftw}{\mathbf{\widetilde w}}
\newcommand{\bfte}{\mathbf{\widetilde e}}
\newcommand{\laplace}{\Delta}

\newcommand{\mat}{\partial^{\bullet}}

\newcommand{\dif}{\text{d}}
\newcommand{\eps}{\varepsilon}

\newcommand{\nb}{\nabla}

\newcommand{\spn}{\textnormal{span}}

\def \to {\rightarrow}

\newcommand{\V}{V}

\newcommand{\phin}{\varphi^\n}
\newcommand{\phiH}{\varphi^H}
\newcommand{\phiw}{\varphi^{\V}}

\newcommand{\phiwn}{\varphi^{z}}

\newcommand{\K}{\bfK}

\newcommand{\tM}{\wt\bfM}
\newcommand{\tA}{\wt\bfA}
\newcommand{\tK}{\wt\bfK}

\newcommand\Q{Q}
\newcommand\Qh{Q_h}

\newcommand\andquad{\quad\hbox{ and }\quad}
\newcommand\forquad{\quad\hbox{ for }\quad}

\newcommand{\xs}{\bfx_\ast}

\newcommand{\eu}{\bfe_u}
\newcommand{\ev}{\bfe_v}
\newcommand{\ex}{\bfe_x}
\newcommand{\ew}{\bfe_w}

\newcommand{\doteu}{\dot\bfe_u}

\newcommand{\dotex}{\dot\bfe_x}

\newcommand{\du}{\bfd_u}
\newcommand{\dv}{\bfd_v}
\newcommand{\dw}{\bfd_w}

\newcommand{\wt}{\widetilde}
\newcommand{\n}{\nu}

\newcommand{\dof}{N}

\begin{document}

\title[Error estimates for BDF discretizations of Willmore flow]{Error estimates for $A$-stable backward difference full discretizations of Willmore flow of closed surfaces}


\author*[1]{\fnm{Nils} \sur{Bullerjahn}}\email{bullerja@math.uni-paderborn.de}
\equalcont{These authors contributed equally to this work.}

\author[1]{\fnm{Bal\'azs} \sur{Kov\'acs}}%\email{balazs.kovacs@math.uni-paderborn.de}
\equalcont{These authors contributed equally to this work.}


\affil*[1]{\orgdiv{Institute of Mathematics}, \orgname{Paderborn University}, \orgaddress{\street{Warburgerstr.~100}, \city{Paderborn}, \postcode{33098}, \state{NRW}, \country{Germany}}}

%\affil[2]{\orgdiv{Department}, \orgname{Organization}, \orgaddress{\street{Street}, \city{City}, \postcode{10587}, \state{State}, \country{Country}}}

%\affil[3]{\orgdiv{Department}, \orgname{Organization}, \orgaddress{\street{Street}, \city{City}, \postcode{610101}, \state{State}, \country{Country}}}


\abstract{A proof of optimal-order $H^1$-norm error estimates is given for $A$-stable backward difference full discretizations (of order 1 and 2) of Willmore flow for closed two-dimensional surfaces. The numerical method discretizes a coupled system of evolution equations by evolving surface finite elements of polynomial degree at least two in space and backward difference method of order 1 or 2 in time. The convergence analysis is based on a stability analysis, based on energy estimates exploiting the anti-symmetric structure of the second-order system, in combination with Dahlquist's $G$-stability and the multiplier techniques of Nevanlinna and Odeh, with a new upper bound in the spirit of Dahlquist. Numerical experiments illustrate and complement the theoretical results.}

\keywords{Willmore flow, geometric evolution equations, evolving surface finite elements, backward difference method (BDF), stability, convergence analysis, energy estimates}


\pacs[MSC Classification]{% as examples
	35R01, % PDEs on manifolds
	53C44, % Geometric evolution equations (mean curvature flow, Ricci flow, etc.)
	65M60, % Finite element, Rayleigh-Ritz and Galerkin methods for initial value and initial-boundary value problems involving PDEs
	65M15, % Error bounds for initial value and initial-boundary value problems involving PDEs
	65M12. % Stability and convergence of numerical methods for initial value and initial-boundary value problems involving PDEs
	}

\maketitle

\section{Introduction}

In this paper we present a fully discrete numerical scheme for the Willmore flow of closed two-dimensional surfaces, by evolving surface finite elements in space of polynomial degree $k\geq2$ and $A$-stable backward difference formulae, i.e.~BDF methods of order $1$ and $2$, in time. We establish optimal-order fully discrete $H^1$-norm error estimates between the numerical approximation and a sufficiently regular solution, see Theorem~\ref{Theorem:MainConvergence}, extending the semi-discrete results of \cite{KovacsLiLubich2021} to the fully discrete setting. 

The Willmore flow is defined as the $L^2$ gradient flow of a surface for the elastic bending energy or Willmore energy, see \cite{Willmore1965,Willmore1993}. The Willmore flow arises in many applications, such as, modelling liquid bilayers \cite{Helfrich1973}, biomembranes \cite{ElliottStinner2010}, vesicles \cite{BGN2015}, regularization of phase-field systems \cite{ChenLowengrubShenWang2018}, and in the analysis of curvature on surfaces; see \cite{MarquesNeves2016} proving the Willmore conjecture. 

There is an extensive literature on numerical methods for Willmore flow \emph{of surfaces}. Numerical schemes based on evolving surface finite elements were proposed in \cite{MayerSimonett,MayerSimonett_SI}, \cite{Rusu2005}, \cite{Dziuk2008}, \cite{BGN2007a} and \cite{BGN2008} based on different variational formulations, indicating convergence in practice by numerical simulations. Recently the authors in \cite{KovacsLiLubich2021} have shown semi-discrete optimal-order convergence, by a novel formulation based on a coupled system of geometric evolution equations. 

In this paper we extend the results of \cite{KovacsLiLubich2021}. We study the full discretization of the weak formulation, obtained from the geometric evolution equations, using evolving surface finite elements of polynomial degree $k\geq2$ in space and $A$-stable backward difference methods (BDF methods) in time, $q=1,2$. The geometric evolution equations here are slightly \emph{different} from the ones derived and used in \cite{KovacsLiLubich2021}: a subtle yet impactful change yields positivity of a critical nonlinear part of the equation, see Remark~\ref{rem:zHChange}. We prove optimal-order fully discrete error estimates by performing a stability and consistency analysis in this fully discrete setting, using and extending techniques from \cite{KovacsLiLubich2019,KovacsLiLubich2021} and \cite{BullerjahnKovacs2024}, respectively. To our knowledge there are no fully discrete error estimates available for Willmore flow of closed two-dimensional surfaces.

The main issue in this paper is proving stability, hence bounding the errors in terms of consistency defects and errors in starting values. We follow the approach of the semi-discrete proof \cite{KovacsLiLubich2021}, combining multiple energy estimates exploiting the \emph{anti-symmetric} structure of the second-order system for the geometric quantities. The adaptation of the energy estimates to the fully discrete setting uses a method combining results from the $G$-stability theory of Dahlquist \cite{Dahlquist1978} and the multiplier technique of Nevanlinna and Odeh \cite{NevanlinnaOdeh1981}. For PDEs, this approach was first used for linear parabolic differential equations on evolving surfaces in \cite{LubichMansourVenkataraman2013}, testing with the errors, and later on for abstract quasi-linear parabolic problems in \cite{AkrivisLubich2015}. Fully discrete energy estimates -- testing with the discrete time derivatives of the errors -- were developed for mean curvature flow in \cite{KovacsLiLubich2019}. We use these techniques in a similar way as in \cite{BullerjahnKovacs2024}, where the Cahn--Hilliard equation with dynamic boundary conditions was analyzed, which has an analogous \emph{anti-symmetric} structure. The resulting scheme of energy estimates is sketched in Figure~\ref{fig:EnergyEstimates}. In addition we use technical lemmas relating different finite element surfaces that were shown in \cite{KovacsLiLubichPower2017} and \cite{KovacsLiLubich2019}. The main new difficulty in these estimates comes from estimating a term on the right-hand side, for which we establish a novel estimate in the spirit of Dahlquist's $G$-stability result, but in the opposite direction, see Lemma~\ref{Lemma:DahlquistUpper}. Finally, we combine the obtained estimates for the geometric evolution equations as in \cite{KovacsLiLubich2019} with estimates for the velocity law, to obtain the stability result.

The consistency analysis. i.e. proving estimates for the defects and their discrete derivatives, uses the semi-discrete consistency results from \cite{KovacsLiLubich2021} to establish the fully discrete consistency bounds, following the lines of the fully discrete consistency analysis in \cite{BullerjahnKovacs2024}, with the adaptation to evolving surface finite elements as in \cite{KovacsLiLubich2019}.

The paper is structured as follows: Section~\ref{Sec:Evolution Eq for Willmore flow}, introduces basic notations and geometric concepts and establishes the \emph{slightly modified} evolution equations for the mean curvature and the normal vector along Willmore flow. In Section~\ref{Sec:Full Discretization} we introduce the fully discrete scheme by evolving surface finite elements and the BDF methods, and in Section~\ref{Sec:MainResult} we present the optimal-order fully discrete convergence result. In Section~\ref{Sec:Stability} we establish the stability estimates, and Section~\ref{Sec:Consistency} is devoted to the consistency analysis. In Section~\ref{Sec:Numerical Results} we present some numerical experiments to illustrate and complement our theoretical results.

\section{Evolution equations for Willmore flow} \label{Sec:Evolution Eq for Willmore flow}

\subsection{Basic notions and notation}%\footnote{This subsection is taken verbatim from \cite{MCF}.} }
\label{subsection: basic notions}

{\it (The text of this preparatory subsection is taken verbatim from \cite[Section~2.1]{KovacsLiLubich2021}.)} 
We consider the evolving two-dimensional closed surface 
$$
\Ga(t) = \{ X(p,t) \,:\, p \in \Ga^0 \}
$$
as the image of a smooth mapping $X \colon \Ga^0\times [0,T]\to \R^3$ such that $X(\cdot,t)$ is an embedding for every $t$. Here, $\Ga^0$ is a smooth closed initial surface, and $X(p,0)=p$. 
In view of the subsequent numerical discretization, it is convenient to think of $X(p,t)$ as the position at time $t$ of a moving particle with label $p$, and of $\Ga(t) $ as a collection of such particles. 
To indicate the dependence of the surface on~$X$, we will write
$$
\Ga(t) = %\Ga[X_t] =
\Ga[X(\cdot,t)] , \quad\hbox{ or briefly}\quad \Ga[X]
$$
when the time $t$ is clear from the context. The {\it velocity} $v(x,t)\in\R^3$ at a point $x=X(p,t)\in\Gamma(t)$  equals
\begin{equation} \label{velocity}
\partial_t X(p,t)= v(X(p,t),t).
\end{equation}
For a known velocity field  $v$,  the position $X(p,t)$ at time $t$ of the particle with label $p$ is obtained by solving the ordinary differential equation \eqref{velocity} from $0$ to $t$ for a fixed $p$.

For a function $u(x,t)$ ($x\in \Gamma(t)$, $0\le t \le T$) we denote the {\it material derivative} (with respect to the parametrization $X$) as
$$
\mat u(x,t) = \frac{\D}{\D t} u(X(p,t),t) \quad\hbox{ for } \ x=X(p,t).
$$
For the following notions, see the review \cite{DeckelnichDziukElliott2005Acta} or \cite[Appendix~A]{Ecker2012} or any textbook on differential geometry. 
On any regular surface $\Gamma\subset\R^3$, we denote by $\nabla_{\Ga}u \colon \Gamma\to\R^3$ the  {\it tangential gradient} of a function $u:\Gamma\to\R$, and in the case of a vector-valued function $u=(u_1,u_2,u_3)^\top \colon \Gamma\to\R^3$, we let
$\nabla_{\Ga} u = (\nabla_{\Ga}u_1, \nabla_{\Ga}u_2, \nabla_{\Ga}u_3)$. We thus use the convention that the gradient of $u$ has the gradient of the components as column vectors. 
We denote by $\nabla_{\Ga} \cdot f$ the {\it surface divergence} of a vector field $f$ on $\Gamma$, 
and by %$\laplace_{\Ga[X]} u=\nabla_{\Ga[X]}\cdot \nabla_{\Ga[X]}u$
$\laplace_{\Ga} u=\nabla_{\Ga}\cdot \nabla_{\Ga}u$ the {\it Laplace--Beltrami operator} applied to $u \colon \Gamma\to\R$. In the case of a
vector-valued function $u=(u_1,u_2,u_3)^\top \colon \Gamma\to\R^3$, we set componentwise $\laplace_{\Ga} u = (\laplace_{\Ga} u_1,\laplace_{\Ga} u_2,\laplace_{\Ga} u_3)^\top$. 
(In the case of a tangential vector field $u$, this componentwise Laplace--Beltrami operator differs from the intrinsic definition of the Laplace--Beltrami operator on tangential vector fields.)

%{\color{blue}If $f$ is a general vector field  on $\Gamma$, then we define its surface divergence as 
%$$\nabla_{\Ga} \cdot f:=\nabla_{\Ga} \cdot P_{\rm tan}f ,$$ where $P_{\rm tan}f$ denotes the tangential component of $f$ at each point of the surface.} 

We denote the unit outer normal vector field to $\Gamma$ by $\n \colon \Gamma\to\R^3$. Its surface gradient contains the (extrinsic) curvature data of the surface $\Gamma$. At every $x\in\Gamma$, the matrix of the {\it extended Weingarten map},
$$
A(x)=\nabla_\Gamma \n(x),
$$ 
is a symmetric $3\times 3$ matrix (see, e.g., \cite[Proposition~20]{Walker2015}). Apart from the eigenvalue $0$ with eigenvector $\n(x)$, its other two eigenvalues are
the principal curvatures $\kappa_1$ and $\kappa_2$ at the point $x$ on the surface. They determine the fundamental quantities
%\begin{align}
%\label{Def-H-A2}
$$
H:={\rm tr}(A)=\kappa_1+\kappa_2, \qquad 
|A|^2 = \kappa_1^2 +\kappa_2^2 ,
$$
%\end{align}
where $|A|$ denotes the Frobenius norm of the matrix $A$.
Here, $H$ is called the {\it mean curvature} (as in most of the literature, we do not put a factor 1/2). 
%{\it Gaussian curvature} is
%$$
%K:=  \kappa_1 \kappa_2 = \half (H^2 - |A|^2) .
%$$

\subsection{The system of equations used for discretization}
\label{section:evolution equation system}

The semi-discretisation in \cite{KovacsLiLubich2021} and the full discretization herein are both based on a coupled fourth-order system, first derived in \cite[Section~2]{KovacsLiLubich2021}, for the geometric quantities of Willmore flow, reformulated as a \emph{system} of non-linear second-order parabolic equations on the surface coupled to ordinary differential equations for the surface positions via the velocity law. The analysis herein requires a slight modification of this system.

\subsubsection{The original system of \cite{KovacsLiLubich2021}}
The coupled system below was originally derived in \cite[Section~2.3]{KovacsLiLubich2021} and used for the semi-discretization therein: 
The system of second-order parabolic equations for the normal vector $\nu\colon \Ga \to \R^3$, the mean curvature $H \colon \Ga \to \R$ and the two new variables $z\colon \Ga \to \R^3$, $V\colon \Ga \to \R$, where $V$ is the normal velocity of the surface, is given by
\begin{subequations}
\label{eq:Willmore evolution equations - system - original}
\begin{align}
	\matdev H =&\ -\Delta_{\Ga[X]} V - |A|^2 V, \label{eq:Willmore evolution equations - system - original a} \\
	V=&\ \Delta_{\Ga[X]} H + Q, \\
	\matdev \nu =&\ -\Delta_{\Ga[X]} z + (HA-A^2) z \nonumber \\
	&\ + |\nabla_{\Ga[X]} H|^2\nu - 2(\nabla_{\Ga[X]}\cdot (A\nabla_{\Ga[X]}H)) \nu - A^2 \nabla_{\Ga[X]} H - \nabla_{\Ga[X]}Q, \label{eq:Willmore evolution equations - system - original c} \\
	z=&\ \Delta_{\Ga[X]} \nu + |A|^2 \nu,
\end{align}
\end{subequations} 
where $A$ denotes the extended Weingarten map, $A(x)=\nabla_{\Ga[X]} \nu(x)$, and $Q=-\frac12 H^3 + |A|^2H$. 

\subsubsection{A modified coupled system}
\label{section:modified system}

The stability analysis of the full discretisation requires us to make a \emph{crucial but subtle change} to the evolution equation \eqref{eq:Willmore evolution equations - system - original c}, to achieve a non-linear term which has a sign.
Namely we used the equality $z=\nabla_{\Ga[X]} H$, see \cite[equations~(2.9), (2.10d)]{KovacsLiLubich2021}, to partially eliminate the dependence on $z$ on the right-hand side of \eqref{eq:Willmore evolution equations - system - original c}, and to render the remaining part to be negative, similarly to \eqref{eq:Willmore evolution equations - system - original a}.

The system of second-order parabolic equations for the normal vector $\nu\colon \Ga \to \R^3$, the mean curvature $H \colon \Ga \to \R$ and the two auxiliary variables $z\colon \Ga \to \R^3$, $V\colon \Ga \to \R$, is given by
\begin{subequations}
\label{eq:Willmore evolution equations - system}
% (a)
%	\begin{align}
	%		\matdev H =&\ -\Delta_{\Ga[X]} V - |A|^2 V, \label{eq:Willmore evolution equations - system a}\\
	%		V=&\ \Delta_{\Ga[X]} H + Q, \\
	%		\matdev \nu =&\ -\Delta_{\Ga[X]} z +(HA-2A^2) \nabla_{\Ga[X]}H \nonumber \\
	%		&\  + |\nabla_{\Ga[X]} H|^2\nu - 2(\nabla_{\Ga[X]}\cdot (A\nabla_{\Ga[X]}H)) \nu - \nabla_{\Ga[X]}Q, \label{eq:Willmore evolution equations - system c}\\
	%		z=&\ \Delta_{\Ga[X]} \nu + |A|^2 \nu,
	%	\end{align}
% (b)
\begin{align}
	\matdev H =&\ -\Delta_{\Ga[X]} V - |A|^2 V, \label{eq:Willmore evolution equations - system a}\\
	V=&\ \Delta_{\Ga[X]} H + Q, \\
	\matdev \nu =&\ -\Delta_{\Ga[X]} z - A^2 z \nonumber \\
	&\ +(HA-A^2) \nabla_{\Ga[X]}H  + |\nabla_{\Ga[X]} H|^2\nu - 2(\nabla_{\Ga[X]}\cdot (A\nabla_{\Ga[X]}H)) \nu - \nabla_{\Ga[X]}Q, \label{eq:Willmore evolution equations - system c}\\
	z=&\ \Delta_{\Ga[X]} \nu + |A|^2 \nu,
\end{align}
\end{subequations} 
with the same notions as before: $A(x)=\nabla_{\Ga[X]} \nu(x)$ and $Q=-\frac12 H^3 + |A|^2H$. 

%\begin{rem} \label{rem:zHChange}
%We have made one subtle change to the evolution equations developed in \cite[Section~2.3]{KovacsLiLubich2021}. Namely we used the equality $z=\nabla_{\Ga[X]} H$, see e.g. \cite[equations~(2.9),(2.10d)]{KovacsLiLubich2021}, to eliminate the dependence on $z$ on the right-hand side of \eqref{eq:Willmore evolution equations - system c}. This will become important in the error analysis, since the term $(HA-A^2)$ has no sign compared to the analogous term $|A|^2$ in \eqref{eq:Willmore evolution equations - system a}.
%\end{rem}

Both systems of second-order equations is coupled to the equations for position and velocity:
\begin{subequations} \label{eq:PositionVelocity}
\begin{align}
	\partial_t X =&\ v \circ X, \\
	v=&\ V\nu.
\end{align}
\end{subequations}

This change will become important in the error analysis, since the term $(HA-A^2)$ in \eqref{eq:Willmore evolution equations - system - original c} has no sign compared to the analogous terms $-|A|^2$ and $-A^2$ in \eqref{eq:Willmore evolution equations - system a} and \eqref{eq:Willmore evolution equations - system c}. Each being the non-linear factors in front of the auxiliary variables $V$ and $z$, respectively:
\begin{align*}
\text{non-linearities for $V$ and $z$ in \eqref{eq:Willmore evolution equations - system - original}}: \qquad
&\ \left\{
\begin{aligned}
	- |A|^2 V & , \\
	+ (HA-A^2) z & ;
\end{aligned}
\right. \\
% (a)
%	\text{non-linearities for $V$ and $z$ in \eqref{eq:Willmore evolution equations - system}}:  \qquad
%	&\ \left\{
%	\begin{aligned}
	%		- |A|^2 V & , \\
	%		\hphantom{+ (HA-2A^2)}0 & .
	%	\end{aligned}
%	\right. 
% (b)
\text{non-linearities for $V$ and $z$ in \eqref{eq:Willmore evolution equations - system}}:  \qquad
&\ \left\{
\begin{aligned}
	- |A|^2 V & , \\
	\hphantom{+ (HA)}-A^2 z & .
\end{aligned}
\right. 
\end{align*}

The numerical discretization is based on a weak formulation of \eqref{eq:Willmore evolution equations - system} with \eqref{eq:PositionVelocity}: we search for $(H,V,\n,z)$ such that %, with the abbreviation $\Q = \tfrac12 H^3 - |A|^2 H$ for the cubic term,
\begin{subequations}
\label{eq:weak form}
\begin{align}
	%		% velocity v
	%		\label{eq:weak form - v}
	%		& \int_{\Ga[X]} \!\!\! \nb_{\Ga[X]} v \cdot  \nb_{\Ga[X]} \phiv + 
	%		\int_{\Ga[X]} \!\!\! v \cdot \phiv 
	%		\nonumber \\ &
	%		= \int_{\Ga[X]} \!\!\! \nb_{\Ga[X]}( \V \n) \cdot \nb_{\Ga[X]}\phiv + 
	%		\int_{\Ga[X]} \!\!\! \V \n \cdot \phiv , \quad \andquad \\[4mm]
	% mean curvature H
	\label{eq:weak form - H}
	&	 \int_{\Ga[X]} \!\!\! \mat H \, \phiH - \int_{\Ga[X]} \!\!\! \nb_{\Ga[X]} \V \cdot  \nb_{\Ga[X]} \phiH = -\int_{\Ga[X]} \!\!\! |A|^2 \,  \V \, \phiH ,
	\\[1mm]
	% normal velocity V
	\label{eq:weak form - V}
	& \int_{\Ga[X]} \!\!\! \V \phiw + \int_{\Ga[X]} \!\!\! \nb_{\Ga[X]} H \cdot \nb_{\Ga[X]} \phiw = \int_{\Ga[X]} \!\!\!  \Q \, \phiw ; 
	\\[4mm]
	% normal vector \nu
	% (a)
	%	\label{eq:weak form - nu}
	%	&	 \int_{\Ga[X]} \!\!\! \mat \n \cdot \phin - \int_{\Ga[X]} \!\!\! \nb_{\Ga[X]} z \cdot \nb_{\Ga[X]} \phin 
	%	=  \int_{\Ga[X]} \! (H A - 2A^2) \nb_{\Ga[X]} H \cdot \phin \nonumber \\
	%	& \qquad\qquad\quad		
	%	+  \int_{\Ga[X]} ( |\nb_{\Ga[X]} H|^2 \nu ) \cdot \phin \nonumber \\
	%	%		& \qquad\qquad\quad		
	%	%		+ 2 \int_{\Ga[X]} \!\!\! (A\nbg H) \cdot \nbg(\n \cdot \phin) 
	%	%		- 2 \int_{\Ga[X]} \!\!\! (\n \cdot \phin) H \,(\nu \cdot A\nbg H) % kb: this term is zero
	%	& \qquad\qquad\quad		
	%	+ 2 \int_{\Ga[X]} \!\!\! (A\nb_{\Ga[X]} H) \cdot (\nb_{\Ga[X]} \phin \n ) \nonumber \\
	%	& \qquad\qquad\quad		
	%	+  \int_{\Ga[X]} \!\!\! \Q \,\nb_{\Ga[X]} \cdot \phin - \int_{\Ga[X]} \!\!\! \Q H \,\nu \cdot \phin, \\
	% normal vector \nu
	% (b)
	\label{eq:weak form - nu}
	&	 \int_{\Ga[X]} \!\!\! \mat \n \cdot \phin - \int_{\Ga[X]} \!\!\! \nb_{\Ga[X]} z \cdot \nb_{\Ga[X]} \phin 
	=  - \int_{\Ga[X]} \! A^2 z \cdot \phin 
	\nonumber \\
	& \qquad\qquad\quad		
	+  \int_{\Ga[X]} \! (H A + A^2) \nb_{\Ga[X]} H \cdot \phin \nonumber \\
	& \qquad\qquad\quad		
	+  \int_{\Ga[X]} ( |\nb_{\Ga[X]} H|^2 \nu ) \cdot \phin \nonumber \\
	%		& \qquad\qquad\quad		
	%		+ 2 \int_{\Ga[X]} \!\!\! (A\nbg H) \cdot \nbg(\n \cdot \phin) 
	%		- 2 \int_{\Ga[X]} \!\!\! (\n \cdot \phin) H \,(\nu \cdot A\nbg H) % kb: this term is zero
	& \qquad\qquad\quad		
	+ 2 \int_{\Ga[X]} \!\!\! (A\nb_{\Ga[X]} H) \cdot (\nb_{\Ga[X]} \phin \n ) \nonumber \\
	& \qquad\qquad\quad		
	+  \int_{\Ga[X]} \!\!\! \Q \,\nb_{\Ga[X]} \cdot \phin - \int_{\Ga[X]} \!\!\! \Q H \,\nu \cdot \phin, \\
	% auxiliary variable z
	\label{eq:weak form - z}
	& \int_{\Ga[X]} \!\!\! z \cdot \phiwn + \int_{\Ga[X]} \!\!\! \nb_{\Ga[X]} \n \cdot \nb_{\Ga[X]} \phiwn = \int_{\Ga[X]} \!\!\! |A|^2 \n \cdot \phiwn ,
\end{align}
\end{subequations}
for all test functions $\phiH \in H^1(\Ga[X])$, $\phiw \in H^1(\Ga[X])$, and $\phin \in H^1(\Ga[X])^3$, $\phiwn \in H^1(\Ga[X])^3$. Here,
we use the Sobolev space 
$H^1(\Ga)=\{ u \in L^2(\Gamma)\,:\, \nb_\Gamma u \in L^2(\Gamma) \}$. The system \eqref{eq:weak form} and \eqref{eq:PositionVelocity} is complemented with the initial data $X^0$, $\n^0$ and $H^0$.

\section{Full discretization by evolving finite elements and linearly implicit backward difference formulae} \label{Sec:Full Discretization}

\subsection{Evolving surface finite elements}

{\it (The text of this preparatory section is taken almost verbatim from \cite[Section~3]{KovacsLiLubich2021}.)} We formulate the evolving surface finite element (ESFEM) discretization for the velocity law coupled with evolution equations on the evolving surface, following the description in \cite{KovacsLiLubichPower2017,KovacsLiLubich2019}, which is based on \cite{Dziuk1988} and \cite{Demlow2009}. We use triangular finite elements on the surface and continuous piecewise polynomial basis functions of degree~$k$, as defined in \cite[Section 2.5]{Demlow2009}.

We triangulate the given smooth initial surface $\Gamma^0$ by an admissible family of triangulations $\mathcal{T}_h$ of decreasing maximal element diameter $h$; see \cite{DziukElliott2013} for the notion of an admissible triangulation, which includes quasi-uniformity and shape regularity. 
For a given triangulation of the initial surface $\Gamma^0$, 
%For a momentarily fixed $h$, 
we denote by $\bfx^0 $  the vector in $\R^{3\dof}$ that collects all nodes $p_j$ $(j=1,\dots,\dof)$ of the initial triangulation. By piecewise polynomial interpolation of degree $k$, the nodal vector defines an approximate surface $\Gamma_h^0$ that interpolates $\Gamma^0$ in the nodes $p_j$. We will evolve the $j$th node in time, denoted $x_j(t)$ with initial condition $x_j(0)=p_j$, and collect the nodes at time $t$ in a column vector% in $\R^{3\dof}$,
$$
\bfx(t) \in \R^{3\dof}. %= (x_1(t),\dots,x_\dof(t)) ^\top 
$$
We just write $\bfx$ for $\bfx(t)$ when the dependence on $t$ is not important.

By piecewise polynomial interpolation on the  plane reference triangle that corresponds to every
curved triangle of the triangulation, the nodal vector $\bfx$ defines a closed surface denoted by $\Gamma_h[\bfx]$. We can then define globally continuous finite element {\it basis functions}
$$
\phi_i[\bfx] \colon \Gamma_h[\bfx]\to\R, \qquad i=1,\dotsc,\dof,
$$
which have the property that on every triangle their pullback to the reference triangle are polynomials of degree $k$, and which satisfy at the node $x_j$
\begin{equation*}
\phi_i[\bfx](x_j) = \delta_{ij} \quad  \text{ for all } i,j = 1,  \dotsc, \dof .
\end{equation*}
These functions span the finite element space on $\Gamma_h[\bfx]$,
\begin{equation*}
S_h[\bfx] = S_h(\Gamma_h[\bfx])=\spn\big\{ \phi_1[\bfx], \phi_2[\bfx], \dotsc, \phi_\dof[\bfx] \big\} .
\end{equation*}
For a finite element function $u_h\in S_h[\bfx]$, the tangential gradient $\nabla_{\Gamma_h[\bfx]}u_h$ is defined piecewise on each element.

The discrete surface at time $t$ is parametrized by the initial discrete surface via the map $X_h(\cdot,t):\Gamma_h^0\to\Gamma_h[\bfx(t)]$ defined by
$$
X_h(p_h,t) = \sum_{j=1}^\dof x_j(t) \, \phi_j[\bfx(0)](p_h), \qquad p_h \in \Gamma_h^0,
$$
which has the properties that $X_h(p_j,t)=x_j(t)$ for $j=1,\dots,\dof$, that  $X_h(p_h,0) = p_h$ for all $p_h\in\Gamma_h^0$, and
$$
\Gamma_h[\bfx(t)]=\Gamma[X_h(\cdot,t)],
$$
where the right-hand side equals $\{ X_h(p_h,t) \,:\, p_h \in \Ga_h^0 \}$ like in Section~\ref{subsection: basic notions}.

The {\it discrete velocity} $v_h(x,t)\in\R^3$ at a point $x=X_h(p_h,t) \in \Gamma[X_h(\cdot,t)]$ is given by
$$
\partial_t X_h(p_h,t) = v_h(X_h(p_h,t),t).
$$
In view of the transport property of the basis functions  \cite{DziukElliott2013},
$$
\frac\D{\D t} \Bigl( \phi_j[\bfx(t)](X_h(p_h,t)) \Bigr) =0 ,
$$
%which by integration from $0$ to $t$ yields
%$$
%	\phi_j[\bfx(t)](X_h(p_h,t)) = \phi_j[\bfx(0)](p_h).
%$$
the discrete velocity equals, for $x \in \Gamma_h[\bfx(t)]$,
$$
v_h(x,t) = \sum_{j=1}^\dof v_j(t) \, \phi_j[\bfx(t)](x) \qquad \hbox{with } \ v_j(t)=\dot x_j(t),
$$
where the dot denotes the time derivative $\D/\D t$. 
Hence, the discrete velocity $v_h(\cdot,t)$ is in the finite element space $S_h[\bfx(t)]$, with nodal vector $\bfv(t)=\dot\bfx(t)$.
%Although it would be possible to define the basis functions using the formula $\phi_j[\bfx\t](X_h(p_h,t)) = \phi_j[\bfx(0)](p_h)$ (obtained by integrating the transport property), such a formulation would involve pullbacks, which are avoided here.
%
%??? $X_h$ ist ja mittels der Basisfunktionen definiert. Ich wuerde das weglassen. ???
%

The {\it discrete material derivative} of a finite element function $u_h(x,t)$ with nodal values $u_j(t)$ is
$$
\mat_h u_h(x,t) = \frac{\D}{\D t} u_h(X_h(p_h,t)) = \sum_{j=1}^\dof \dot u_j(t)  \phi_j[\bfx(t)](x)  \quad\text{at}\quad x=X_h(p_h,t).
$$

%[kb: Would be nice to shorten up a little.] 

%We will determine an approximate normal vector $\n_h$ and an approximate mean curvature $H_h$ as finite element functions: for $x \in \Gamma_h[\bfx(t)]$,
%\begin{align*}
%	\n_h(x,t) &= \sum_{j=1}^\dof \n_j(t) \, \phi_j[\bfx(t)](x) , \\
%	H_h(x,t) &= \sum_{j=1}^\dof H_j(t) \, \phi_j[\bfx(t)](x).
%\end{align*}
%Note that these finite element functions are {\it not} the normal vector and the mean curvature of the discrete surface $\Gamma_h[\bfx(t)]$. 
%For a curved triangle $E_h\subset\Gamma_h[\bfx]$, let $\varphi_{_{E_h}}: E\rightarrow E_h$ denote the one-to-one and onto map (a polynomial of degree $k$) defined on the reference triangle $\widetilde E$. 

\subsection{ESFEM spatial semi-discretization}
\label{subsection:semi-discretization}
A {\it preliminary} finite element spatial semi-discretization of the second-order parabolic coupled system \eqref{eq:weak form} together with
the velocity and position equations \eqref{eq:PositionVelocity}
reads as follows: Find the unknown nodal vector $\bfx(t)\in \R^{3\dof}$ of the finite element surface parametrization $X_h(\cdot,t)\in S_h[\bfx^0]^3$ and the unknown finite element velocity $v_h(\cdot,t)\in S_h[\bfx(t)]^3$, and the finite element functions $H_h(\cdot,t)\in S_h[\bfx(t)]$, $\V_h(\cdot,t)\in S_h[\bfx(t)]$, and $\n_h(\cdot,t)\in S_h[\bfx(t)]^3$, $z_h(\cdot,t)\in S_h[\bfx(t)]^3$ such that
\begin{subequations}
\label{eq:xh-vh}
\begin{align}
	\label{eq:xh}
	\partial_t X_h(p_h,t) = v_h(X_h(p_h,t),t), \qquad p_h\in\Ga_h^0,
\end{align}
with
\begin{align}\label{eq:vh}
	v_h = \widetilde I_h (V_h \n_h),
\end{align}
\end{subequations}
where $\widetilde I_h=\widetilde I_h[\bfx]:C(\Ga_h[\bfx])\to S_h(\Ga_h[\bfx])$ denotes the  finite element interpolation operator on the discrete surface $\Ga_h[\bfx]$.

The functions $H_h$, $V_h$, $\n_h$, $z_h$ are determined by the ESFEM semi-discretization of \eqref{eq:weak form},
denoting by $A_h = \frac12 (\nb_{\Ga_h[\bfx]} \n_h + (\nb_{\Ga_h[\bfx]} \n_h)^\top)$ the symmetric part of $\nb_{\Ga_h[\bfx]} \n_h$, 
%by $K_h = \half ( H_h^2 - |A_h|^2 )$ the approximate Gaussian curvature 
and by $\Qh = -\tfrac12 H_h^3 + |A_h|^2 H_h$ the cubic term, 
\begin{subequations}
\label{eq:semidiscrete weak form}
\begin{align}
	%		% velocity v_h
	%		& \int_{\Ga_h[\bfx]} \!\!\!\! \nb_{\Ga_h[\bfx]} v_h \cdot  \nb_{\Ga_h[\bfx]} \phiv_h + 
	%		\int_{\Ga_h[\bfx]} \!\!\!\! v_h \cdot \phiv_h 
	%		\nonumber \\ &
	%		= \int_{\Ga_h[\bfx]} \!\!\!\! \nb_{\Ga_h[\bfx]}( \V_h \n_h) \cdot \nb_{\Ga_h[\bfx]}\phiv_h + 
	%		\int_{\Ga_h[\bfx]} \!\!\!\! \V_h \n_h \cdot \phiv_h , \quad \\[3mm]
	% mean curvature H_h
	&	 \int_{\Ga_h[\bfx]} \!\!\!\! \mat_h H_h \, \phiH_h - \int_{\Ga_h[\bfx]} \!\!\!\! \nb_{\Ga_h[\bfx]} \V_h \cdot  \nb_{\Ga_h[\bfx]} \phiH_h  = - \int_{\Ga_h[\bfx]} \!\!\!\! | A_h |^2 \,  \V_h \, \phiH_h , \\[8pt]
	% normal velocity V_h
	\label{Vh-prelim}
	& \int_{\Ga_h[\bfx]} \!\!\!\! \V_h \phiw_h + \int_{\Ga_h[\bfx]} \!\!\!\! \nb_{\Ga_h[\bfx]} H_h \cdot \nb_{\Ga_h[\bfx]} \phiw_h =  \int_{\Ga_h[\bfx]} \!\!\! \Qh \phiw_h ;
	\\[16pt]
	% normal vector \nu_h 
	% (a)
	%	\label{nuh-prelim}
	%	&	 \int_{\Ga_h[\bfx]} \!\!\!\! \mat_h \n_h \cdot \phin_h 
	%	- \int_{\Ga_h[\bfx]} \!\!\!\!  \nb_{\Ga_h[\bfx]} z_h \cdot \nb_{\Ga_h[\bfx]} \phin_h  
	%	= \int_{\Ga_h[\bfx]} \!  (H_h A_h - 2A_h^2)\nb_{\Ga_h[\bfx]} H_h \cdot \phin_h \nonumber \\
	%	&\hspace{72pt}  		
	%	+  \int_{\Ga_h[\bfx]} \big( |\nb_{\Ga_h[\bfx]} H_h|^2 \nu_h \big) \cdot \phin_h \nonumber \\
	%	& \hspace{72pt}  		
	%	+ 2 \int_{\Ga_h[\bfx]} \!\!\! (A_h\nb_{\Ga_h[\bfx]} H_h) \cdot (\nb_{\Ga_h[\bfx]} \phin_h\, \n_h ) \nonumber \\
	%	&\hspace{72pt}  
	%	+ \int_{\Ga_h[\bfx]} \!\!\!  \Qh \,\nb_{\Ga_h[\bfx]} \cdot \phin_h - \int_{\Ga_h[\bfx]} \!\!\! \Qh H_h \, \nu_h\cdot \phin_h, \\[8pt]
	% normal vector \nu_h 
	% (b)
	\label{nuh-prelim}
	&	 \int_{\Ga_h[\bfx]} \!\!\!\! \mat_h \n_h \cdot \phin_h 
	- \int_{\Ga_h[\bfx]} \!\!\!\!  \nb_{\Ga_h[\bfx]} z_h \cdot \nb_{\Ga_h[\bfx]} \phin_h  
	= - \int_{\Ga_h[\bfx]} \! A_h^2 z_h \cdot \phin_h \nonumber \\
	&\hspace{72pt} 
	\int_{\Ga_h[\bfx]} \!  (H_h A_h + A_h^2)\nb_{\Ga_h[\bfx]} H_h \cdot \phin_h \nonumber \\
	&\hspace{72pt} 
	+  \int_{\Ga_h[\bfx]} \big( |\nb_{\Ga_h[\bfx]} H_h|^2 \nu_h \big) \cdot \phin_h \nonumber \\
	& \hspace{72pt}  		
	+ 2 \int_{\Ga_h[\bfx]} \!\!\! (A_h\nb_{\Ga_h[\bfx]} H_h) \cdot (\nb_{\Ga_h[\bfx]} \phin_h\, \n_h ) \nonumber \\
	&\hspace{72pt}  
	+ \int_{\Ga_h[\bfx]} \!\!\!  \Qh \,\nb_{\Ga_h[\bfx]} \cdot \phin_h - \int_{\Ga_h[\bfx]} \!\!\! \Qh H_h \, \nu_h\cdot \phin_h, \\[8pt]
	% auxiliary variable z_h
	\label{zh-prelim}
	& \int_{\Ga_h[\bfx]} \!\!\!\! z_h \cdot \phiwn_h + \int_{\Ga_h[\bfx]} \!\!\!\! \nb_{\Ga_h[\bfx]} \n_h \cdot \nb_{\Ga_h[\bfx]} \phiwn_h = \int_{\Ga_h[\bfx]} \!\!\!\! | A_h |^2 \n_h \cdot \phiwn_h ,
\end{align}
\end{subequations}
for all $\phiH_h \in S_h[\bfx(t)]$, $\phiw_h \in S_h[\bfx(t)]$, $\phin_h \in S_h[\bfx(t)]^3$, and $\phiwn_h \in S_h[\bfx(t)]^3$. 

The initial values for the nodal vector $\bfx$ are taken as the positions of the nodes of the triangulation of the given initial surface $\Gamma^0$.
The initial data for $H_h$ and $\n_h$ are determined by Lagrange interpolation of $H^0$ and $\n^0$, respectively. 

We note that the finite element functions $\n_h$ and $H_h$ are {\it not} the normal vector and mean curvature of the discrete surface $\Ga_h[\bfx]$.

\subsection{Matrix--vector formulation}
\label{subsection:DAE}
We collect the nodal values of $v_h \in S_h[\bfx(t)]^3$, $H_h \in S_h[\bfx(t)]$, $\n_h \in S_h[\bfx(t)]^3$, $V_h \in S_h[\bfx(t)]$, and $z_h \in S_h[\bfx(t)]^3$ in column vectors  $\bfv=(v_j) \in \R^{3N}$, $\bfH=(H_j)\in\R^\dof$, $\bfn=(\n_j) \in \R^{3N}$, $\boldsymbol{\V} = ({\V}_j) \in \R^\dof$, and $\bfz = (z_j) \in \R^{3N}$, respectively, and denote 
\begin{equation*}
\bfu=\left(\begin{array}{c}
	\bfH\\ 
	\bfn
\end{array}\right) \in \R^{4N} \andquad
\bfw=\left(\begin{array}{c}
	\boldsymbol{\V}\\ 
	\bfz
\end{array}\right) \in \R^{4N} .
\end{equation*}

We define the surface-dependent mass matrix $\bfM(\bfx) \in \R^{N \times N}$ and stiffness matrix $\bfA(\bfx) \in \R^{N \times N}$ %(cf.~\cite[Section~2.5]{KovacsLiLubichPower2017}) 
on the surface determined by the nodal vector $\bfx$:
\begin{equation*}
\begin{aligned}
	\bfM(\bfx)\vert_{ij} =&\ \int_{\Ga_h[\bfx]} \! \phi_i[\bfx] \phi_j[\bfx] , \\
	\bfA(\bfx)\vert_{ij} =&\ \int_{\Ga_h[\bfx]} \! \nb_{\Ga_h[\bfx]} \phi_i[\bfx] \cdot \nb_{\Ga_h[\bfx]} \phi_j[\bfx] , 
\end{aligned}
\qquad i,j = 1,  \dotsc,\dof ,
\end{equation*}
with the finite element nodal basis functions $\phi_j[\bfx] \in S_h[\bfx]$.
We further let, for $d=3$ or $4$ (with the identity matrices $I_d \in \R^{d \times d}$) 
$$
\bbM^{[d]}(\bfx)= I_d \otimes \bfM(\bfx), \quad
\bbA^{[d]}(\bfx)= I_d \otimes \bfA(\bfx).
%	\quad
%	\bfK^{[d]}(\bfx) = I_d \otimes \bigl( \bfM(\bfx) + \bfA(\bfx) \bigr).
$$
When no confusion can arise, we write $\bfM(\bfx)$ for $\bfM^{[d]}(\bfx)$ and $\bfA(\bfx)$ for $\bfA^{[d]}(\bfx)$.

We recall that we denote by $|A_h|^2$ the squared Frobenius norm of $A_h = \tfrac12 \big(\nb_{\Ga_h[\bfx]} \n_h + (\nb_{\Ga_h[\bfx]} \n_h)^\top\big)$. We define the surface- and $| A_h |^2$-dependent matrix $\bfF_1(\bfx,\bfu) \in \R^{N \times N}$ by
\begin{equation}
\label{eq:alpha mass matrix - F_1}
\bfF_1(\bfx,\bfu)\vert_{ij} = \int_{\Ga_h[\bfx]} \! | A_h |^2 \,\phi_i[\bfx] \phi_j[\bfx] , 
\qquad i,j = 1,  \dotsc,\dof .
\end{equation}
and similarly
\begin{equation}
\label{eq:alpha mass matrix - F_2}
\bfF_2(\bfx,\bfu)\vert_{ij} = \int_{\Ga_h[\bfx]} \! A_h^2 \,\phi_i[\bfx] \cdot \phi_j[\bfx] , 
\qquad i,j = 1,  \dotsc,3\dof .
\end{equation}
We then define the block diagonal matrix
\begin{equation}
\label{eq:bfF definition}
\bfF(\bfx,\bfu) = 
% (a)
%	\left(\begin{array}{cc}
	%		\bfF_1(\bfx,\bfu) & 0\\ 
	%		0 &   0 
	%	\end{array}\right) 
% (b)
\left(\begin{array}{cc}
	\bfF_1(\bfx,\bfu) & 0\\ 
	0 & \bfF_2(\bfx,\bfu) 
\end{array}\right) 
\in \R^{4N \times 4N} .
\end{equation}

We define the non-linear functions $\bff(\bfx,\bfu), \bfg(\bfx,\bfu) \in\R^{4N}$ by
\begin{equation*}
\bff(\bfx,\bfu) =
\left(\begin{array}{c}
	0  \\ 
	\bff_{2}(\bfx,\bfu)  
\end{array}\right) \andquad
\bfg(\bfx,\bfu) =
\left(\begin{array}{c}
	\bfg_{1}(\bfx,\bfu) \\ 
	\bfg_{2}(\bfx,\bfu)  
\end{array}\right) ,
\end{equation*}
with $\bff_{2}(\bfx,\bfu)\in \R^{3N}$, $\bfg_{1}(\bfx,\bfu)\in \R^{N}$ and  $\bfg_{2}(\bfx,\bfu)\in \R^{3N}$, given by  
\begin{equation*}
\begin{aligned}
	\bff_{2}(\bfx,\bfu)\vert_{j+(\ell - 1)N} \: &= 
	\int_{\Ga_h[\bfx]} \big( |\nb_{\Ga_h[\bfx]} H_h|^2 \nu_h + (H_h A_h + A_h^2) \nb_{\Ga_h[\bfx]} H_h \big)_\ell \, \phi_j[\bfx] 
	\nonumber \\
	& \hspace{12pt}  		
	+ 2 \int_{\Ga_h[\bfx]} \!\!\! (A_h\nb_{\Ga_h[\bfx]} H_h) \cdot (\nb_{\Ga_h[\bfx]} \phi_j[\bfx]\, (\n_h )_\ell )
	\\
	& \hspace{12pt} 
	+\int_{\Gamma_h[\bfx]} \!\! \Qh \, (\nb_{\Ga_h[\bfx]})_\ell \phi_j[\bfx] 
	%		\\
	%		 &\quad\,
	-\int_{\Ga_h[\bfx]} \!\!\! \Qh H_h \, (\nu_h)_\ell \, \phi_j[\bfx] , \\
	\bfg_{1}(\bfx,\bfu)\vert_j \qquad\quad\ &=  \int_{\Gamma_h[\bfx]}\!\! \Qh \, \phi_j[\bfx] ,
	\\
	\bfg_{2}(\bfx,\bfu)\vert_{j+(\ell-1)N} &=\int_{\Gamma_h[\bfx]}\!\! | A_h |^2 \,(\n_h)_\ell \,  \, \phi_j[\bfx] ,
	\\
	%		\bfg(\bfx,\bfu,\bfw)\vert_{j+(\ell-1)N} &= -\int_{\Gamma_h[\bfx]} \!\!\!\!\!   \V_h  (\n_h  )_\ell \, \phi_j[\bfx]
	%		-\int_{\Gamma_h[\bfx]} \!\!\!\!\!  \nb_{\Ga_h[\bfx]} (\V_h  (\n_h  )_\ell)\cdot  \nb_{\Ga_h[\bfx]}\phi_j[\bfx] ,
\end{aligned}
%	\qquad (j = 1,  \dotsc, N,\ \ell=1,2,3) .
\end{equation*}
for $j = 1, \dotsc, N$ and $\ell=1,2,3$.

%, and $\bfK(\bfx)$ for $\bfK^{[d]}(\bfx)$.
%, and
%$\| \cdot \|_{H^1(\Gamma)}$ for $\| \cdot \|_{H^1(\Gamma)^4}$, etc.

The position and velocity equations \eqref{eq:xh-vh} are equivalent to 
\begin{subequations}
\label{eq:matrix-form-X-v}
\begin{align}
	\label{eq:matrix-form-X}
	\dot\bfx &= \bfv,
	\\
	\label{eq:matrix-form-v}
	\bfv &= \bfV \bullet \bfn,
\end{align}
\end{subequations}
where $\bullet$ denotes the componentwise product of vectors, $\bfV \bullet \bfn = (V_j\n_j)$.
With the notation introduced above, the equations \eqref{eq:semidiscrete weak form} can be written in the following matrix--vector form, where we recall that $\bfu = (\bfH;\bfn)$ and $\bfw = (\bfV;\bfz)$: 
\begin{subequations}
\label{eq:matrix--vector form}
\begin{align}
	%		\label{eq:matrix-form-v}
	%		\bfK^{[3]}(\bfx)\bfv = &\ \bfg(\bfx,\bfu,\bfw) , 
	%		\\
	\label{eq:matrix-form-u}
	\bfM^{[4]}(\bfx)\dot\bfu - \bfA^{[4]}(\bfx)\bfw = &\ -\bfF(\bfx,\bfu)\bfw + \bff(\bfx,\bfu) , \\
	\label{eq:matrix-form-w}
	\bfM^{[4]}(\bfx)\bfw + \bfA^{[4]}(\bfx)\bfu = &\ \bfg(\bfx,\bfu) .
\end{align}
\end{subequations}

% (a)
%\begin{rem}
%\label{rem:zHChange}
%	The matrix--vector formulation of \eqref{eq:Willmore evolution equations - system - original}, see Section~3.3 in \cite{KovacsLiLubich2021}, is formally identical to \eqref{eq:matrix--vector form}. The corresponding $\bfF$ matrix however reads (with $\bfF_1$ from above):
%	\begin{align*}
%		&\ \bfF^{\textnormal{old}}(\bfx,\bfu) = \left(\begin{array}{cc}
	%			- \bfF_1(\bfx,\bfu) & 0\\ 
	%			0 &   \bfF_2^{\textnormal{old}}(\bfx,\bfu) 
	%		\end{array}\right) \in \R^{4N \times 4N} \\
%		&\ \text{with} \quad
%		\bfF_2^{\textnormal{old}}(\bfx,\bfu)\vert_{ij} = \int_{\Ga_h[\bfx]} \! (H_h A_h - A_h^2) \,\phi_i[\bfx] \cdot \phi_j[\bfx] ,
%	\end{align*}
%	compared to \eqref{eq:bfF definition}:
%	\begin{equation*}
%		\bfF(\bfx,\bfu) = 
%		\left(\begin{array}{cc}
	%			\bfF_1(\bfx,\bfu) & 0\\ 
	%			0 &   0 
	%		\end{array}\right) \in \R^{4N \times 4N}
%	\end{equation*}
%	
%	Note that the original matrix $\bfF^{\textnormal{old}}(\bfx,\bfu)$ is indefinite due to lower-right block. 
%	The new block matrix $\bfF(\bfx,\bfu)$ is positive semi-definite, since $|A_h|^2 \geq 0$.
%	By the change made in the evolution equations \eqref{eq:Willmore evolution equations - system}, compared to \eqref{eq:Willmore evolution equations - system - original}, the matrix $\bfF(\bfx,\bfu)$ now becomes non-negative definite. This will be an essential property used in the stability proof.
%\end{rem}
% (b)
\begin{rem}
\label{rem:zHChange}
The matrix--vector formulation of \eqref{eq:Willmore evolution equations - system - original}, see Section~3.3 in \cite{KovacsLiLubich2021}, is formally identical to \eqref{eq:matrix--vector form}. 

The first block of the matrix $\bfF$ is identical (up to a change in sign convention):
\begin{align*}
	\bfF_1^{\textnormal{old}}(\bfx,\bfu) = - \bfF_1(\bfx,\bfu) .
\end{align*}
The second block of the matrix $\bfF$ however changes substantially:
\begin{align*}
	\bfF_2^{\textnormal{old}}(\bfx,\bfu)\vert_{ij} = &\ \int_{\Ga_h[\bfx]} \! (H_h A_h - A_h^2) \,\phi_i[\bfx] \cdot \phi_j[\bfx] , \\
	\bfF_2(\bfx,\bfu)\vert_{ij} = &\ \int_{\Ga_h[\bfx]} \! A_h^2 \,\phi_i[\bfx] \cdot \phi_j[\bfx] .
\end{align*}

Note that the original block matrix $\bfF^{\textnormal{old}}(\bfx,\bfu)$ is indefinite due to the lower-right block. By the change made in the evolution equations \eqref{eq:Willmore evolution equations - system}, compared to \eqref{eq:Willmore evolution equations - system - original}, the matrix $\bfF(\bfx,\bfu)$ now becomes positive semi-definite, since $|A_h|^2 \geq 0$ and $A_h^2$ has non-negative eigenvalues $\kappa_{1,h}^2, \kappa_{2,h}^2 \geq 0$. This will be an essential property used in the stability analysis.
\end{rem}

\subsection{Lifts}
\label{subsec:lifts}

%{\it (The text of this preparatory section is taken almost verbatim from \cite[Section~3]{KovacsLiLubich2021}.)} 

We need to compare functions on the {\it exact surface} $\Gamma(t)=\Gamma[X(\cdot,t)]$ with functions on the {\it discrete surface} $\Gamma_h(t)=\Gamma_h[\bfx(t)]$, as well as on the {\it interpolated surface} $\Gamma_h^*(t)=\Gamma_h[\xs(t)]$ with $\xs(t)|_j=X(p_j,t)$.%, where $\xs(t)$ denotes the nodal vector collecting the grid points $x_j^*(t)=X(p_j,t)$ on the exact surface.

Following \cite[Section~3.5]{KovacsLiLubich2021} we will use two different lift operations: For any finite element function $w_h:\Gamma_h(t)\to\R^m$ ($m=1$ or 3), with nodal vector $\bfw \in \R^{mN}$, we associate the finite element function on the interpolated surface $\Gamma_h^*(t)$: 
$$
\widehat w_h  = \sum_{j=1}^N \bfw_j \phi_j[\xs(t)] \in S_h[\xs]^m .
$$ 

%This can be further lifted to a function on the exact surface $\Ga(t)$ by using the \emph{lift operator} $^\ell$, which was introduced for linear and higher-order surface approximations in \cite{Dziuk1988} and \cite{Demlow2009}, respectively. 

%The lift operator $^\ell$ maps a function on the interpolated surface $\Gamma_h^*$ to a function on the exact surface $\Gamma$, provided that $\Gamma_h^*$ is sufficiently close to $\Gamma$.
%The exact regular surface $\Gamma$ can be represented,  in some neighbourhood of the surface,  by a smooth signed distance function $d \colon \R^3 \times [0,T] \to \R$, cf.~\cite[Section~2.1]{DziukElliott2013}, such that $\Gamma$ is the zero level set of~$d$  (i.e., $x\in \Gamma$ if and only if $d(x)=0$), with negative values of $d$ in the interior.  
This can be further lifted to a function on the smooth exact surface $\Ga(t)$ by using the \emph{lift operator} $^\ell$, see \cite{Dziuk1988} and \cite{Demlow2009}. Using the distance function representation of $\Ga(t)$ (i.e., $d \colon \R^3 \times [0,T] \to \R$ such that $\Ga(t) = \{x \in \R^3 \,:\, d(x,t) = 0\}$) and the uniquely defined closest point projection ($y = x - \nu(y) d(x)$), the lift of a continuous function $\eta_h \colon \Ga_h^* \to \R^m$ is defined as
$\eta_{h}^{\ell}(y) := \eta_h(x)$ for $x\in\Ga_h^*$, for every $x\in \Ga_h^*$ with projection $y=y(x)\in\Ga$.

%where for every $x\in \Ga_h^*$ the point $y=y(x)\in\Ga$ is uniquely defined via
%$%\begin{equation*}
%%\label{eq: lift defining equation}
%y = x - \nu(y) d(x).
%$%\end{equation*}

The composed lift $^L$ from finite element functions on the discrete surface $\Gamma_h(t)$ to functions on the exact surface $\Gamma(t)$ via the interpolated surface $\Gamma_h^*(t)$ is denoted by 
$$
w_h^L = (\widehat w_h)^\ell.
$$

\subsection{Linearly implicit full discretization}

Let $\tau>0$ be the time step size, and $q\tau \leq t_n:=n\tau \leq T$ be a uniform partition of the time interval $[0,T]$. We assume the starting values $\bfu^0,\dots, \bfu^{q-1} \in \R^{4N}$ and $\bfx^0,\dots,\bfx^{q-1} \in \R^{3N}$ to be given. We discretize in time by the linearly implicit backward differentiation formulae (BDF methods), where the discretized time derivative and the extrapolation for the non-linear terms are, respectively, given by 
\begin{align}
\partial_q^\tau \bfw^n =\mathbf{\dot w}^n := \frac1\tau \sum_{j=0}^q \delta_j \bfw^{n-j}, \andquad \widetilde \bfw^n := \sum_{j=0}^{q-1} \gamma_j \bfu^{n-1-j}, \forquad n\geq q,
\end{align}
for a vector $\bfw$, where the coefficients are given by the expressions
\begin{subequations}\label{eq:BdfCoeff}
\begin{align}
	\label{eq:BDF generating function}
	\delta(\zeta)=&\sum_{j=0}^q \delta_j \zeta^j=\sum_{l=1}^q \frac{1}{l} (1-\zeta)^l, \qquad
	\gamma(\zeta)=&\sum_{j=0}^{q-1}\gamma_j \zeta^j=\frac{(1-(1-\zeta)^q)}{\zeta}.
\end{align}
\end{subequations}

The $q$-step BDF methods are of order $q$, they are 

\begin{center}
\begin{minipage}{0.742\linewidth}
	\begin{itemize}
		\item[--] $A$-stable for $q=1,2$,
		\item[--] $A(\alpha_q)$-stable for $q=3,\dotsc,6$,
	\end{itemize}
\end{minipage}
\end{center}

with $\alpha_3 =86.03^\circ$, $\alpha_4=73.35^\circ$, $\alpha_5=51.84^\circ$,  and $\alpha_6 = 17.84^\circ$, respectively, and unstable for $q \geq 7$, see, e.g.,  \cite[Section~V.2]{HairerWanner1996}, \cite[Section~2.3]{AkrivisLubich2015}, while see \cite{AkrivisKatsoprinakis2020} for the exact $\alpha_q$ values. 

The starting values can be precomputed using either a lower order method with smaller time step size, or a Runge--Kutta method of suitable order.

Then the numerical scheme gives approximations $\bfu^n=(\bfH^n,\bfn^n)$, $\bfw^n=(\bfV^n,\bfz^n)$, $\bfx^n$ and $\bfv^n$ by the fully discrete system of linear equations
\begin{subequations} \label{eq:NumericalScheme4order}
\begin{align}
	\bfM^{[4]}(\bftx^n) \bfdu^n - \bfA^{[4]} (\bftx^n) \bfw^n =&\ -\bfF(\bftx^n,\bftu^n) \bftw^n + \bff(\bftx^n,\bftu^n), \label{eq:NumericalScheme4orderA} \\
	\bfM^{[4]}(\bftx^n) \bfw^n + \bfA^{[4]}(\bftx^n) \bfu^n =&\ \bfg(\bftx^n,\bftu^n). \label{eq:NumericalScheme4orderB}
\end{align}
\end{subequations}
with the equations for position and velocity,
\begin{subequations} \label{eq:NumericalSchemeODE}
\begin{align}
	\bfdx^n =&\ \bfv^n, \\
	\bfv^n =&\ \bfV^n \bullet \bfn^n.
\end{align}
\end{subequations}

From nodal vectors, e.g., $\bfx^n=(x_j^n)$, we obtain approximations of the corresponding geometric variables, $X(\cdot,t_n) \approx X_h^n(\cdot) = \sum_{j=1}^N x_j^n \phi_j[\bfx(0)]$, and similarly for other variables.

%From the vectors $\bfx^n=(x_j^n)$ and $\bfv^n=(v_j^n)$ we obtain position and velocity approximations to $X(\cdot,t_n), \text{id}_{\Ga[X(\cdot,t_n)]},v(\cdot,t_n)$ as
%\begin{align}
%X_h^n(p_h) =&\ \sum_{j=1}^N x_j^n \phi_j[\bfx(0)](p_h) \forquad p_h \in \Ga_h^0, \\
%x_h^n =&\ \text{id}_{\Ga[X_h^n]}, \\
%v_h^n(x) =&\ \sum_{j=1}^N v_j^n \phi_j[\bfx^n](x) \forquad x \in \Ga_h[\bfx^n],
%\end{align}
%and from vectors $\bfu^n = (u_j^n)$ with $u_j^n = (H_j^n,\nu_j^n) \in \R \times \R^3$ we obtain approximations to the mean curvature and normal vector at time $t_n$ as
%\begin{align}
%\nu_h^n(x) =&\ \sum_{j=1}^N \nu_j^n \phi_j[\bfx^n](x) \forquad x \in \Ga_h[\bfx^n], \\
%H_h^n(x) =&\ \sum_{j=1}^N H_j^n \phi_j[\bfx^n](x) \forquad x \in \Ga_h[\bfx^n].
%\end{align}

\subsection{Modified numerical scheme}

In order to obtain the optimal-order of convergence with our error analysis, we need the initial values of $\bfx$, $\bfu$ and $\bfw$ to be $\calO(h^k)$ close to the Ritz projection of the exact initial values $\bfx_\ast$, $\bfu_\ast$ and $\bfw_\ast$ in the $H^1$-norm. The initial values $\bfx^i$ and $\bfu^i$ for $i=0,\dots,q-1$ are freely chosen, but since $\bfw$ is given by an algebraic equation \eqref{eq:NumericalScheme4orderB}, the initial values are not freely chosen, which makes it necessary to modify the equation, as it is done in the semi-discrete error analysis, see \cite[Section~3.4]{KovacsLiLubich2019}. The shift we introduce is analogous to the shift introduced in the fully discrete scheme of the related Cahn--Hilliard equation with dynamic boundary conditions in \cite[Section~4.1]{BullerjahnKovacs2024}: 

For $i=0,\dots,q-1$, we define $\bfd_w^i$ as the nodal vector of the semi-discrete defect $d_w(t_i)$ at time $t_i=i\tau$, see \cite[equation~(5.15b)]{KovacsLiLubich2021}. Then we define the shift as
\begin{equation} \label{eq:DefVartheta}
\begin{aligned}
	\bfvartheta^i=&\ -\bfA(\bfx_\ast^i) \bfu_\ast^i + \bfA(\bfx^i) \bfu^i + \bfg(\bfx_\ast^i,\bfu_\ast^i) - \bfg(\bfx^i,\bfu^i) + \bfM(\bfx_\ast^i) \bfd_w^i, \\
	\bfvartheta^n=&\ \bfvartheta^{q-1}, 
\end{aligned}
\end{equation}
for $i=0,\dots,q-1$ and $n\geq q$, and modify \eqref{eq:NumericalScheme4orderB} to
\begin{align}
\bfM(\bftx^n) \bfw^n + \bfA(\bftx^n) \bfu^n =&\ \bfg(\bftx^n,\bftu^n) + \bfvartheta^n. \label{eq:NumericalScheme4orderBMod}
\end{align}

If we now consider this algebraic equation in the initial time steps $i=0,\dots,q-1$, we extend the extrapolation by the exact value and obtain
\begin{align*}
\bfM(\bfx^i) \bfw^i =&\ -\bfA(\bfx^i) \bfu^i + \bfg(\bfx^i,\bfu^i) + \bfvartheta^i \\
=&\ -\bfA(\bfx_\ast^i) \bfu_\ast^i + \bfg(\bfx_\ast^i,\bfu_\ast^i) + \bfM(\bfx_\ast^i) \bfd_w^i \\
=&\ \bfM(\bfx_\ast^i) \bfw_\ast^i,
\end{align*}
by the choice of the initial values in $\bfx$ and $\bfu$, and by the definition of the semi-discrete defect.

\section{Fully discrete optimal-order error estimates} \label{Sec:MainResult}

In this section we present our main result of optimal-order convergence of the full discretization by evolving surface finite elements of polynomial degree at least $2$ and a $A$-stable backward difference formulae, i.e.~BDF methods of order $1$ and $2$, to a sufficiently regular solution of the Willmore flow of closed surfaces \eqref{eq:PositionVelocity}--\eqref{eq:weak form}.

\begin{thm} \label{Theorem:MainConvergence}
Consider the $A$-stable BDF--ESFEM full discretization \eqref{eq:NumericalScheme4order}--\eqref{eq:NumericalSchemeODE} of the Willmore flow problem \eqref{eq:PositionVelocity}--\eqref{eq:weak form} with modification \eqref{eq:NumericalScheme4orderBMod}, using evolving surface finite elements of polynomial degree $k\geq2$ and linearly implicit BDF time discretization of order $q = 1,2$. Suppose that the problem admits an exact solution $(X,v,\nu,H,z,V)$ that is sufficiently differentiable on the time interval $t \in [0,T]$, and that the flow map $X(\cdot,t) \colon \Ga^0 \to \Ga(t) \subset \R^3$ is non-degenerate so that $\Ga(t)$ is a regular surface on the time interval $t\in [0,T]$. 

Then, there exist $h_0>0$ and $\tau_0>0$ such that for all mesh sizes $h\leq h_0$ and time step sizes $\tau\leq \tau_0$ satisfying the step size restriction $\tau^q \leq C_0 h^k$ (where $C_0>0$ can be chosen arbitrarily), the following error bounds for the lifts of the discrete position, velocity, normal vector and mean curvature hold over the exact surface. Provided that the starting values are $\calO(h^k\tau^{1/2}+\tau^{q+1/2})$ accurate in the $H^1$ norm at time $t_i=i\tau$ for $i=0,\dots,q-1$, we have at time $t_n=n\tau\leq T$:
\begin{align}
	\| (x_h^n)^L - \text{id}_{\Ga(t_n)} \|_{H^1(\Ga(t_n))^3} \leq C(h^k+\tau^q), \\
	\| (v_h^n)^L - v(\cdot,t_n) \|_{H^1(\Ga(t_n))^3} \leq C(h^k+\tau^q), \\
	\| (\nu_h^n)^L - \nu(\cdot,t_n) \|_{H^1(\Ga(t_n))^3} \leq C(h^k+\tau^q), \\
	\| (H_h^n)^L - H(\cdot,t_n) \|_{H^1(\Ga(t_n))} \leq C(h^k+\tau^q), \\
	\| (z_h^n)^L - \nabla_{\Ga(t)} H(\cdot,t_n) \|_{H^1(\Ga(t_n))^3} \leq C(h^k+\tau^q), \\
	\| (V_h^n)^L - V(\cdot,t_n) \|_{H^1(\Ga(t_n))} \leq C(h^k+\tau^q),
\end{align}
and also 
\begin{align}
	\| (X_h^n)^\ell - X(\cdot,t_n) \|_{H^1(\Ga_0)^3} \leq C(h^k+\tau^q),
\end{align}
where the constant $C$ is independent of $h$, $\tau$ and $n$ with $n\tau\leq T$, but depends on bounds of higher derivatives of the solution $(X,v,\nu,H)$ of the Willmore flow and on the length $T$ of the time interval, and on $C_0$.
\end{thm}

Sufficient regularity assumptions are the following: uniformly in $t \in [0,T]$,
\begin{align*}
&X(\cdot,t) \in H^{k+1}(\Ga^0), \partial_t^jX(\cdot,t) \in H^1(\Ga^0) \quad (j=1,\dots,q+1), \\
&v(\cdot,t) \in H^{k+1}(\Ga(X(\cdot,t))), \\
\text{for }u=(H,\nu),\quad &u(\cdot,t), \partial^{(j)} u(\cdot,t) \in W^{k+1,\infty}(\Ga(X(\cdot,t)))^4 \quad (j=1,\dots,q+3), \\
\text{for }w=(V,z),\quad &w(\cdot,t), \partial^{(j)} w(\cdot,t) \in W^{k+1,\infty}(\Ga(X(\cdot,t)))^4 \quad (j=1,\dots,q+2).
\end{align*}
%\NBon These conditions are safe, we need q+1 derivatives for discrete derivative, q derivatives for extrapolation and 2 additional derivatives for discrete derivative of defect terms. So q+1 on X, 0 on v, q+3 on u, q+2 on w and we only need \NBoff

%\begin{rem}
%	(i) The $h_0$ depends on geometric approximation results. The requirement on the smallness of both parameters is merely technical. 
%	
%	(ii) The step size restriction is needed to establish an $W^{1,\infty}$-bound on the error in the stability analysis. The restriction can be even weakened to $\tau^q \leq C_0 h^{3/2+\eps_0}$, where $\eps_0,C_0 > 0$ can be chosen arbitrarily, as is apparent from the assumptions in Proposition~\ref{Prop:StabilityEst}.
%\end{rem}

The error estimates are shown by clearly separating the stability and consistency analysis.

In order to prove the stability bound, we follow the general scheme of the semi-discrete stability proof in \cite[Section~5]{KovacsLiLubich2021}, and adapt crucial techniques from the fully discrete stability analysis of \cite[Section~6]{BullerjahnKovacs2024}. 
We derive two sets of energy estimates by exploiting the anti-symmetric structure of \eqref{eq:NumericalScheme4order}, see Figure~\ref{fig:EnergyEstimates}, being the time discrete counterpart of Figure~2 in \cite{KovacsLiLubich2021}. The fully discrete energy estimates rely on the $G$-stability theory of Dahlquist \cite[Theorem 3.3]{Dahlquist1978} and the multiplier technique of Nevanlinna and Odeh \cite[Section 2]{NevanlinnaOdeh1981}, in combination with a new Dahlquist-type upper bound Lemma~\ref{Lemma:DahlquistUpper}. We relate different finite element surfaces using the results from \cite{KovacsLiLubichPower2017}. This general approach has already proved to be effective for various partial differential equations, see, e.g., \cite{LubichMansourVenkataraman2013,AkrivisLubich2015,KovacsLiLubich2019,LLG,ContriKovacsMassing2023,Bullerjahn2025}.

The main new difficulty arises from the error term $(\bfde_w^n)^\top \bfF(\bftx^n,\bftu^n) \bfte_w^n$: Following the semi-discrete proof, this should be estimated by a symmetric product rule, similar to terms like $(\bfde^n)^\top \bfM(\bftx^n) \bfe^n$ appearing on the left-hand side of the estimates. Unfortunately, in the fully discrete setting there is only a lower bound available (via $G$-stability and the multiplier technique). In order to resolve this, we establish a new upper bound in the spirit of Dahlquist stability theory, see Lemma~\ref{Lemma:DahlquistUpper}, which requires the non-negativity of $\bfF$. In turn, this explains our modifications in Section~\ref{section:evolution equation system}, see also Remark~\ref{rem:zHChange}. 
%\NBon This new result, Lemma~\ref{Lemma:DahlquistUpper}, restricts us to $A$-stable BDF methods, i.e.~to order $q=1,2$, all other arguments and estimates can be performed for $q=1,\dots,5$, cf.~\cite{KovacsLiLubich2019,BullerjahnKovacs2024}. \NBoff

%In the consistency part we are concerned with the approximation error which comes from the discretization and enters as defect terms in the stability analysis. 
The consistency analysis relies on an interplay of the semi-discrete consistency results from \cite[Section~6]{KovacsLiLubich2021} and the adaptation to fully discrete consistency estimates in \cite[Section~7]{BullerjahnKovacs2024}.%, with the adaptation to evolving surface finite elements as in \cite{KovacsLiLubich2019}.

\section{Stability}\label{Sec:Stability}

\subsection{Preparation: Estimates relating different finite element surfaces}
\label{section:aux}

%In our previous work \cite[Section~4]{KovacsLiLubichPower2017} and \cite[Section~7.1]{MCF} we proved technical results relating different finite element surfaces, which we recapitulate here, {\it taking verbatim the text of \cite[Section~7.1]{MCF} in this preparatory subsection.}
%We use the following setting.

The following technical results relating different finite element surfaces were developed in \cite[Section~4]{KovacsLiLubichPower2017} and \cite[Section~7.1]{KovacsLiLubich2019}.

The  finite element matrices of Section~\ref{subsection:DAE} induce discrete versions of Sobolev norms. Let $\bfx \in \R^{3 N}$ be a nodal vector defining the discrete surface $\Gamma_h[\bfx]$. For any nodal vector $\bfw=(w_j) \in \R^{N}$, with the corresponding finite element function $w_h= \sum_{j=1}^\dof w_j \phi_j[\bfx] \in S_h[\bfx]$, we define the following norms, where
$ \bfK(\bfx) = \bfM(\bfx) + \bfA(\bfx)$ in the third line:
\begin{equation}
\label{eq:def norms}
\begin{aligned}
	%\label{M-L2}
	&  \|\bfw\|_{\bfM(\bfx)}^{2} = \bfw^\top \bfM(\bfx) \bfw = \|w_h\|_{L^2(\Ga_h[\bfx])}^2 , \\
	%\label{A-H1}
	&  \|\bfw\|_{\bfA(\bfx)}^{2} = \bfw^\top \bfA(\bfx) \bfw = \|\nb_{\Ga_h[\bfx]} 	w_h\|_{L^2(\Ga_h[\bfx])}^2 , \\
	%\label{K-H1}
	&  \|\bfw\|_{\bfK(\bfx)}^{2} = \bfw^\top \bfK(\bfx) \bfw = \|w_h\|_{H^1(\Ga_h[\bfx])}^2 .
\end{aligned}
\end{equation}
In the following, when $\bfw\in \R^{dN}$, so that the corresponding finite element function $w_h$ maps into $\R^d$, we write simply $\| w_h \|_{L^2(\Gamma)}$ for $\| w_h \|_{L^2(\Gamma)^d}$ and $\| w_h \|_{H^1(\Gamma)}$ for $\| w_h \|_{H^1(\Gamma)^d}$, respectively.

Let now $\bfx,\bfy \in \R^{3 N}$ be two nodal vectors defining discrete surfaces $\Gamma_h[\bfx]$ and $\Gamma_h[\bfy]$, respectively. 
We  denote the difference by $\bfe= (e_j)=\bfx-\bfy \in \R^{  3  N}$. 
For  $\theta\in[0,1]$, we consider the intermediate surface $\Gamma_h^\theta=\Gamma_h[\bfy+\theta\bfe]$ and the corresponding finite element functions given as
$$
e_h^\theta=\sum_{j=1}^\dof e_j \phi_j[\bfy+\theta\bfe] .
$$

Under the condition that $\eps := \| \nabla_{\Gamma_h[\bfy]} e_h^0 \|_{L^\infty(\Gamma_h[\bfy])}\le \tfrac14$, it was shown that
\begin{equation}
\label{norm-equiv}
\begin{aligned}
	&\text{the norms $\|\cdot\|_{\bfM(\bfy+\theta\bfe)}$ are $h$-uniformly equivalent for $0\le\theta\le 1$,}
	\\
	&\text{and so are the norms $\|\cdot\|_{\bfA(\bfy+\theta\bfe)}$.}
\end{aligned}
\end{equation}

Additionally, for $\bfz\in \R^{dN}$ the following bounds were established:
\begin{equation}
\label{matrix difference bounds}
\begin{aligned}
	\bfw^\top (\bfM(\bfx)-\bfM(\bfy)) \bfz \leq &\ c \, \eps \, \|\bfw\|_{\bfM(\bfy)} \|\bfz\|_{\bfM(\bfy)} , \\[1mm]
	\bfw^\top (\bfA(\bfx)-\bfA(\bfy)) \bfz \leq &\ c \, \eps \, \|\bfw\|_{\bfA(\bfy)} \|\bfz\|_{\bfA(\bfy)} ,
\end{aligned}
\end{equation}
and similarly, using the $L^\infty$ norm of $z_h$ or its gradient and the $L^2$ norm of the gradient of $e_h$:
\begin{equation}
\label{matrix difference bounds e_x}
\begin{aligned}
	\bfw^\top (\bfM(\bfx)-\bfM(\bfy)) \bfz \leq &\ c \, \|\bfw\|_{\bfM(\bfy)} \|\bfe\|_{\bfA(\bfy)} , \\[1mm]
	\bfw^\top (\bfA(\bfx)-\bfA(\bfy)) \bfz \leq &\ c \, \|\bfw\|_{\bfA(\bfy)} \|\bfe\|_{\bfA(\bfy)} .
\end{aligned}
\end{equation}

\subsection{Defects and errors}

We choose reference finite element functions $x_h^*(\cdot,t)$, $v_h^*(\cdot,t)$, $u_h^*(\cdot,t),w_h^*(\cdot,t)$ on the interpolated surface $\Gamma_h[\bfx_\ast(t)]$ with nodal vectors, related to the exact solution $X$, $v$ and $u=(H,\nu)$, $w=(V,z)$ as follows:
%\begin{alignat*}{3}
%	& \text{nodal vector} & \quad & \text{collecting the values of}, \\
%	& \bfx_\ast(t)\in\R^{3N} & \quad & \text{the finite element nodes of $X$}, \\
%	& \bfv_\ast(t)\in\R^{3N} & \quad & \text{the finite element nodes of $v$}, \\
%	& \bfu_\ast(t)=\begin{pmatrix} \bfH_\ast(t) \\ \mathbf{n}_\ast(t) \end{pmatrix} \in\R^{4N} & \qquad & \text{the modified Ritz map of $u=(H,\nu)$}, \\
%	& \bfw_\ast(t)=\begin{pmatrix} \bfV_\ast(t) \\ \mathbf{z}_\ast(t) \end{pmatrix} \in\R^{4N}  & \qquad & \text{the Ritz map of $w=(V,z)$}, 
%\end{alignat*}
\begin{table}[h]
\centering
\begin{tabular}{ll}
	\toprule
	nodal vector & collecting the values of \\
	\midrule
	$\bfx_\ast(t)\in\R^{3N}$ & the nodal values of $X$, \\
	$\bfv_\ast(t)\in\R^{3N}$ & the nodal values of $v$, \\
	$\bfu_\ast(t)=\begin{pmatrix} \bfH_\ast(t) \\ \mathbf{n}_\ast(t) \end{pmatrix} \in\R^{4N}$ & the modified Ritz map of $u=(H,\nu)$ , \\
	$\bfw_\ast(t)=\begin{pmatrix} \bfV_\ast(t) \\ \mathbf{z}_\ast(t) \end{pmatrix} \in\R^{4N}$ & the Ritz map of $w=(V,z)$ .\\
	\bottomrule
\end{tabular}
\end{table}

The Ritz map and the \emph{modified} Ritz map are defined on the interpolated surface $\Gamma_h[\bfx_\ast(t)]$, see \cite[Section~6]{KovacsLiLubich2021}.

We then consider the vectors $\bfx_\ast^n=\bfx_\ast(t_n)$, $\bfv_\ast^n=\bfv_\ast(t_n)$, $\bfu_\ast^n=\bfu_\ast(t_n)$ and $\bfw_\ast^n=\bfw_\ast(t_n)$, which then satisfy the equations \eqref{eq:NumericalSchemeODE} up to some defects $\bfd_x^n$ and $\bfd_v^n$,
\begin{subequations}
\label{eq:matrix-form-X-v-star}
\begin{align}
	\label{eq:matrix-form-X-star}
	\mathbf{\dot x}_\ast^n &= \bfv_\ast^n + \bfd_x^n,
	\\
	\label{dv}
	\bfv_\ast^n &= \mathbf{V}_\ast^n \bullet \mathbf{n}_\ast^n + \dv^n, 
\end{align}
\end{subequations}
and $\bfu_\ast^n$, $\bfw_\ast^n$ satisfy the equations  \eqref{eq:NumericalScheme4order} up to some defects  $\bfd_{\bf u}^n$ and $\bfd_{\bf w}^n$: 
\begin{subequations}
\label{eq:defect vectors}
\begin{align}
	\bfM(\bftx_\ast^n) \bfdu_\ast^n - \bfA (\bftx_\ast^n) \bfw_\ast^n =&\ -\bfF(\bftx_\ast^n,\bftu_\ast^n) \bftw_\ast^n + \bff(\bftx_\ast^n,\bftu_\ast^n) + \bfM(\bftx_\ast^n) \bfd_{\bf u}^n, \\
	\bfM(\bftx_\ast^n) \bfw_\ast^n + \bfA(\bftx_\ast^n) \bfu_\ast^n =&\ \bfg(\bftx_\ast^n,\bftu_\ast^n) + \bfM(\bftx_\ast^n)\bfd_{\bf w}^n .
\end{align}
\end{subequations}
In the following, we simplify the matrix notation and abbreviate, e.g., $\bfM(\bftx^n)$ to $\tM^n$, and $\bfM(\bftx_\ast^n)$ to and $\tM_\ast^n$, etc.

The errors between the nodal values of the numerical solutions and the nodal values of the interpolated exact values are denoted by $\ex^n = \bfx^n - \bfx_\ast^n$, $\ev^n = \bfv^n - \bfv_\ast^n$, $\eu^n = \bfu^n - \bfu_\ast^n$ and $\ew^n = \bfw^n - \bfw_\ast^n$ and their corresponding finite element functions on the interpolated surface $\Ga_h[\bfx_\ast^n]$ are denoted by $e_x^n$, $e_v^n$, $e_u^n$, and $e_w^n$, respectively.

We obtain the error equations by subtracting \eqref{eq:defect vectors} from \eqref{eq:NumericalScheme4order}, and \eqref{eq:matrix-form-X-v-star} from \eqref{eq:NumericalSchemeODE}: for $n\geq q$,
\begin{subequations} 
\label{eq:error equations}
\begin{align}
	\label{eq:error eq - x}
	\dotex^n =&\ \ev^n -\bfd_x^n, \\
	\label{eq:error eq - v}
	\ev^n =&\   \mathbf{V}^n \bullet \mathbf{n}^n- \mathbf{V}_\ast^n \bullet \mathbf{n}_\ast^n - \bfd_v^n,
	\\
	\label{eq:error eq - u}
	\tM^n \doteu^n - \tA^n \ew^n =&\ \bfr_1^n, \\
	\label{eq:error eq - w}
	\tM^n \ew^n + \tA^n \eu^n =&\ \bfr_2^n,
\end{align}
\end{subequations}
where we have abbreviated 
\begin{align*}
\bfr_1^n:=&\ -  \big( \tM^n-\tM_\ast^n \big) \mathbf{\dot u}_\ast^n + \big( \tA^n-\tA_\ast^n \big) \bfw_\ast^n - \big(\bfF(\bftx^n,\bftu^n)\bftw^n - \bfF(\bftx_\ast^n,\bftu_\ast^n)\bftw_\ast^n\big) \nonumber \\
&\ + \big(\bff(\bftx^n,\bftu^n) - \bff(\bftx_\ast^n,\bftu_\ast^n)\big) - \tM_\ast^n\du^n , \\
\bfr_2^n:=&\ -  \big( \tM^n-\tM_\ast^n \big) \bfw_\ast^n - \big( \tA^n-\tA_\ast^n \big) \bfu_\ast^n + \big(\bfg(\bftx^n,\bftu^n) - \bfg(\bftx_\ast^n,\bftu_\ast^n)\big) - \tM_\ast^n\dw^n + \bfvartheta^n.
\end{align*}

\subsection{Results by Dahlquist and Nevanlinna $\&$ Odeh}

In order to derive energy estimates for BDF methods, we rely on important results from the $G$-stability theory of Dahlquist \cite[Theorem~3.3]{Dahlquist1978} and the multiplier technique of Nevanlinna and Odeh \cite[Section~2]{NevanlinnaOdeh1981}. 

%Note that Dahlquist's result is slightly extended to semi-inner products in \cite[Lemma~3.5]{ContriKovacsMassing2023}.
%\begin{lem}[{\cite[Theorem~3.3]{Dahlquist1978}}]
%	\label{Lemma:DahlquistOrig} 
%	Let $\delta (\zeta)= \sum_{j=0}^q \delta_j \zeta^j$ and $\mu(\zeta)=\sum_{j=0}^q \mu_j \zeta^j$ be polynomials of degree at most $q$, and at least one of them of degree $q$, that have no common divisor. Let $\langle\cdot,\cdot\rangle$ denote an inner product on $\mathbb{R}^N$ with associated norm $|\cdot|$. If
%	\begin{align*}
%		\Real \frac{\delta(\zeta)}{\mu(\zeta)} >0 , \qquad \text{ for all } \quad \zeta \in \mathbb{C}, \ |\zeta|<1,
%	\end{align*}
%	then there exists a symmetric positive definite matrix $G=(g_{ij}) \in \mathbb{R}^{q\times q}$ and real numbers $\gamma_0,\dotsc,\gamma_q$ such that, for all $w_0,\dotsc,w_q \in \mathbb{R}^N$, the following identity holds
%	\begin{align*}
%		\Real \Big\langle \sum_{i=0}^q \delta_i w_{q-i}, \sum_{j=0}^q \mu_j w_{q-j} \Big\rangle = \sum_{i,j=1}^q g_{ij} \langle w_i,w_j \rangle - \sum_{i,j=1}^q g_{ij} \langle w_{i-1},w_{j-1} \rangle +\Big|\sum_{i=0}^q \gamma_i w_i \Big|^2.
%	\end{align*}
%\end{lem}
%%
%Additionally, we will need this result for an arbitrary semi-inner product, therefore we use the following extension of the above result to semi-inner products by \cite[Lemma~3.5]{ContriKovacsMassing2023}.
\begin{lem}[{$G$-stability \cite[Theorem~3.3]{Dahlquist1978}}]]
\label{Lemma:Dahlquist} 
Let $\delta (\zeta)= \sum_{j=0}^q \delta_j \zeta^j$ and $\mu(\zeta)=\sum_{j=0}^q \mu_j \zeta^j$ be polynomials of degree at most $q$, and at least one of them of degree $q$, that have no common divisor. Let $(\cdot , \cdot )$ be a semi-inner product on a Hilbert space $H$ with associated semi-norm $|\cdot |$. If
\begin{align*}
	\Real \frac{\delta(\zeta)}{\mu(\zeta)} >0 , \qquad \text{ for all } \quad \zeta \in \mathbb{C}, \ |\zeta|<1,
\end{align*}
then there exists a symmetric positive definite matrix $G=(g_{ij}) \in \mathbb{R}^{q\times q}$ and real numbers $\gamma_0,\dotsc,\gamma_q$ such that, for all $w_0,\dotsc,w_q \in H$, we have
\begin{align*}
	\Real \Big( \sum_{i=0}^q \delta_i w_{q-i}, \sum_{j=0}^q \mu_j w_{q-j} \Big) = \sum_{i,j=1}^q g_{ij} ( w_i,w_j ) - \sum_{i,j=1}^q g_{ij} ( w_{i-1},w_{j-1} ) +  \Big|\sum_{i=0}^q \gamma_i w_i \Big|^2.
\end{align*}
\end{lem}

The application of $G$-stability to the BDF schemes, for $q=1,\dotsc,5$, is ensured by the following result.
\begin{lem}[{Nevanlinna and Odeh multipliers \cite[Section~2]{NevanlinnaOdeh1981}}]
\label{Lemma:MultiplierTechnique}
For $1\leq q\leq 5$, there exists $0 \leq \eta_q <1$ such that for $\delta (\zeta)= \sum_{j=0}^q \frac{1}{j} (1-\zeta)^j$,
\begin{align*}
	\Real \frac{\delta(\zeta)}{1-\eta_q \zeta} >0 , \qquad \text{ for all } \quad \zeta \in \mathbb{C}, \ |\zeta|<1 .
\end{align*}
The classical values of $\eta_q$ from \cite[Table]{NevanlinnaOdeh1981} are respectively found to be 

\begin{center}
	\begin{minipage}{0.742\linewidth}
		\begin{itemize}
			\item[--] $\eta_1=0$, $\eta_2=0$ for $A$-stable BDF methods $q = 1,2$,
			\item[--] $\eta_3=0.0836$, $\eta_4=0.2878$, $\eta_5=0.8160$ for $A(\alpha_q)$-stable BDF methods $q = 2,\dotsc,5$.
		\end{itemize}
	\end{minipage}
\end{center}
\end{lem}
The exact multipliers were computed in \cite{AkrivisKatsoprinakis2015}, while multipliers for BDF6 were derived in \cite{AkrivisChenYuZhou2021}.

%\NBon It is crucial to point out here, that we will only use these lemmas for the $A$-stable case, $q = 1,2$, i.e.~for which the multiplier is $\eta_q = 0$.\NBoff

We thus introduce the $G$-semi-norm associated to the semi-inner product $(\cdot,\cdot)$ on a Hilbert space $H$: Given a collection of vectors $W^n=(w^n,\dotsc,w^{n-q+1}) \subset H^q$, we define
\begin{align*}
|W^n|_{G}^2:= \sum_{i,j=1}^q g_{ij} ( w^{n-i+1},w^{n-j+1} ),
\end{align*}
where $G$ is the symmetric positive definite matrix appearing in Lemma~\ref{Lemma:Dahlquist}. Then with the smallest and largest eigenvalues of $G$, denoted by $\lambda_0$ and $\lambda_1$, we have the inequalities:
\begin{align}
\lambda_0 |w^n|^2 \leq \lambda_0 \sum_{j=1}^q |w^{n-j+1}|^2 \leq |W^n|_G^2 \leq \lambda_1 \sum_{j=1}^q |w^{n-j+1}|^2, \label{eq:GnormEst}
\end{align}
where $|\cdot|$ is the semi-norm on $H$ induced by the semi-inner product $(\cdot,\cdot)$.
Later on, an additional subscript, e.g.~$|\cdot|_{G,\bfA^n}$, specifies which semi-inner product generates the $G$-weighted semi-norm. Note that Dahlquist's result is slightly extended to semi-inner products in \cite[Lemma~3.5]{ContriKovacsMassing2023}.

These results have previously been applied to the error analysis for the BDF time discretization of PDEs when testing the error equation with the error in \cite{LubichMansourVenkataraman2013}, and the discrete time derivative of the error in \cite{KovacsLiLubich2019}.

\subsection{An upper bound in the spirit of Dahlquist}

In order to get the desired upper bound of one particular error term on the right-hand side, coming from $\wt\bfF^n \bfte_w^n$, we would like to treat this as in the semi-discrete error analysis of \cite{KovacsLiLubich2021} by a product rule. 

For the fully discrete error analysis we need to apply a similar technique as for the left-hand side, combining the results from Dahlquist and Nevanlinna $\&$ Odeh. But in contrast to a lower bound in this classical case on the left-hand side, we need here an upper bound on the right-hand side.

We show the following variant of Lemma~\ref{Lemma:Dahlquist}.
\begin{lem}[{Dahlquist-type upper bound}]
\label{Lemma:DahlquistUpper}
Let $q \in \{1,2\}$ and $\delta (\zeta)= \sum_{j=1}^q \frac1j (1-\zeta)^j$ be the generating polynomial of the BDF method, and $\gamma(\zeta)=\frac{(1-(1-\zeta)^q)}{\zeta}$ the generating polynomial of the extrapolation.

Then, there exists a real (diagonal) matrix $B\in \R^{q\times q}$ and real vector $a \in \R^{q+1}$, such that, for all $w_0,\dotsc,w_q \in H$, and an arbitrary symmetric bilinear form $(\cdot,\cdot) \colon H \times H \to \R$ we have:

\begin{align*}
	\Big( \sum_{i=0}^q \delta_i w_{q-i}, \sum_{j=0}^q \gamma_j w_{q-j} \Big) 
	= &\ \sum_{i,j=1}^q B_{ij} ( w_i,w_j ) - \sum_{i,j=1}^q B_{ij} ( w_{i-1},w_{j-1} ) \\
	&\ -  \Big(\sum_{i=0}^q a_i w_i ,\sum_{i=0}^q a_i w_i \Big) .
\end{align*}
\end{lem}
\begin{proof}
We first consider $q=1$: In this case the equation reduces to
\begin{align*}
	(w_1-w_0,w_0)=B_{11} (w_1,w_1) - B_{11} (w_0,w_0) - (a_1 w_1 + a_0 w_0,a_1 w_1 + a_0 w_0),
\end{align*} 
and we obtain a solution by comparison of coefficients, as

$$B = \frac12, \qquad a= \Big( \frac1{\sqrt2}, -\frac1{\sqrt2} \Big).$$

In the case $q=2$, we analogously obtain the solution:
\begin{alignat*}{3}
	B = &\ \begin{bmatrix}
		-\frac14 & 0\\ 
		0 &   \frac34 
	\end{bmatrix} , &\ \qquad &\ a= \Big( \frac{\sqrt3}2, -\sqrt3, \frac{\sqrt3}2 \Big) .
\end{alignat*}
\end{proof}

\begin{rem}[{Upper bound for $q = 1,2$}]
\label{Rem:UpperBound}
This now puts us in a position, where for a \emph{positive semi-definite} bilinear form $\calF$, which we obtained for $\wt\bfF^n$ by changing the evolution equations, see Remark~\ref{rem:zHChange}, Lemma~\ref{Lemma:DahlquistUpper} yields
\begin{align*}
	\mathcal{F} \Big( \sum_{i=0}^q \delta_i w_{q-i}, \sum_{j=0}^q \gamma_j w_{q-j} \Big) \leq \sum_{i,j=1}^q B_{ij} \calF( w_i,w_j ) - \sum_{i,j=1}^q B_{ij} \calF( w_{i-1},w_{j-1} ) .
\end{align*}
\end{rem}

\begin{rem}[{Required upper bound for $q = 3,4,5$}] \label{rem: Upper bound q345}
Higher order BDF methods with $q=3,4,5$ are no longer $A$-stable, but only $A(\alpha_q)$-stable, which results in $\eta_q \neq 0$, cf. Lemma~\ref{Lemma:MultiplierTechnique}. This means that the left-hand side of the upper bound comparable to Remark~\ref{Rem:UpperBound} for higher order BDF methods, would need to be of the form
\begin{align*}
	\mathcal{F} \Big( \sum_{i=0}^q \delta_i w_{q-i}, \sum_{j=0}^q \gamma_j w_{q-j} - \eta_q \sum_{j=0}^q \gamma_j w_{q-j-1} \Big),
\end{align*}
so that we would be able to perform similar estimates in this case. 
Notice that the left-hand side now depends on $w_{-1}$, which means we need a different structure of terms on the right-hand side as well, allowing for cancellation by a telescoping sum. 

To the knowledge of the authors, such an estimate is not known, nevertheless if such an estimate would be available we strongly expect that the stability analysis, and hence the main result, of this paper could be extended to BDF methods of order $q=3,4,5$.
\end{rem}

\subsection{Stability estimate}

We bound the errors in terms of the defects and initial values. The errors will be estimated in the $H^1$ norm on the interpolated surface $\Ga_h[\bfx_\ast^n]$: For a nodal vector $\bfe$ corresponding to a finite element function $e \in S_h(\bfx_\ast^n)$, recalling \eqref{eq:def norms} with $\bfK_\ast^n=\bfM_\ast^n+\bfA_\ast^n$, we have
$\| \bfe \|_{\K(\bfx_\ast^n)}^2 = \bfe \TbfK(\bfx_\ast^n) \bfe = \| e \|_{H^1(\Ga_h[\bfx_\ast^n])}^2$.
The defect terms will either be estimated in the \emph{$H^1$ norm} 
\begin{align}
\| \bfd \|_{\bfK(\bfx_\ast^n)}^2 = \bfd \TbfK(\bfx_\ast^n) \bfd = \| d \|_{H^1(\Ga_h[\bfx_\ast^n])}^2,
\end{align}
or the \emph{weak dual norm}
\begin{align}
\| \bfd \|_{\ast,\bfx_\ast^n}^2 := \bfd \TbfM(\bfx_\ast^n) \bfK(\bfx_\ast^n)^{-1} \bfM(\bfx_\ast^n) \bfd,
\end{align}
for the nodal vector $\bfd$ of a corresponding finite element function $d \in S_h[\bfx_\ast^n]$. By \cite[equation~(5.5)]{KovacsLiLubichPower2017}, this equals the following dual norm:
\begin{align}
\| \bfd \|_{\ast,\bfx_\ast^n}^2 = \| d\|_{H_h^{-1}(\Ga_h[\bfx_\ast^n])}^2:= \sup_{0\not= \varphi_h \in S_h[\bfx_\ast^n]} \frac{\int_{\Ga_h[\bfx_\ast^n]} d \varphi_h}{\|\varphi_h \|_{H^1(\Ga_h[\bfx_\ast^n])}}.
\end{align}

We are now in place to prove the fully discrete stability estimate, which is the heart of our convergence proof.

\begin{prop} \label{Prop:StabilityEst}
%\NBon Consider the linearly implicit BDF time discretization of order $q = 1, 2$. \NBoff

Assume that for step sizes restricted by $\tau^q  \leq C_0 h^k$, there exists $\kappa$ with $1<\kappa \leq k$ such that the defects are bounded by
\begin{equation}\label{eq:ConditionDefectsSatbility}
	\begin{aligned}
		\| \bfd_x^n \|_{\bfK(\bfx_\ast^n)} + \| \bfd_v^n \|_{\bfK(\bfx_\ast^n)} + \| \bfd_u^n \|_{\ast,\bfx_\ast^n} + \|\dot \bfd_u^n\|_{\ast,\bfx_\ast^n}  \leq&\ ch^\kappa, \\
		\| \bfd_w^n \|_{\ast,\bfx_\ast^n} + \|\dot \bfd_w^n\|_{\ast,\bfx_\ast^n} + \|\bfd_w^i\|_{\ast,\bfx_\ast^i} \leq&\ ch^\kappa 
	\end{aligned}
\end{equation}
for $q\tau \leq n\tau \leq T$ and $0\leq i \leq q-1$, and that also the error of the starting values are bounded by 
\begin{align}
	I_h^{q-1}:= \sum_{i=0}^{q-1} \| \bfe_x^i\|_{\bfK(\bfx_\ast^n)}^2 + \| \bfe_u^i\|_{\bfK(\bfx_\ast^n)}^2 + \| \bfe_w^i\|_{\bfK(\bfx_\ast^n)}^2 + \tau \sum_{i=1}^{q-1} \| \partial^\tau \bfe_x^i\|_{\bfK(\bfx_\ast^n)}^2 \leq ch^{2\kappa} , \label{eq:AssInitialValues}
\end{align}
where we use the notational convention $\partial^\tau=\partial_1^\tau$.

Then there exists $h_0>0$ and $\tau_0>0$ such that the following stability estimate holds for all $h \leq h_0$, $\tau \leq \tau_0$, satisfying $\tau^q \leq C_0 h^k$ (where $C_0>$ arbitrary), and $n$ with $q\tau \leq n\tau \leq T$, 
\begin{align}
	\| \bfe_x^n\|_{\bfK(\bfx_\ast^n)}^2 + \| \bfe_v^n\|_{\bfK(\bfx_\ast^n)}^2 + \| \bfe_u^n\|_{\bfK(\bfx_\ast^n)}^2 + \| \bfe_w^n\|_{\bfK(\bfx_\ast^n)}^2 \leq c (I_h^{q-1}+D_h^n), \label{eq:StabilityBound}
\end{align}
where the defects are collected in 
\begin{align}
	D_h^n 
	%	:=&\  \max_{q\leq k \leq n}  \| \bfd_x^k\|_{\bfK(\bfx_\ast^k)}^2 + \max_{q\leq k \leq n}  \| \bfd_v^k\|_{\bfK(\bfx_\ast^k)}^2 + \max_{q\leq k \leq n}  \| \bfd_u^k\|_{\ast,\bfx_\ast^k}^2 + \max_{q\leq k \leq n}  \| \bfd_w^k\|_{\ast,\bfx_\ast^k}^2 \nonumber \\
	:=&\  \max_{q\leq k \leq n} \big( \| \bfd_x^k\|_{\bfK(\bfx_\ast^k)}^2 + \| \bfd_v^k\|_{\bfK(\bfx_\ast^k)}^2 + \| \bfd_u^k\|_{\ast,\bfx_\ast^k}^2 + \| \bfd_w^k\|_{\ast,\bfx_\ast^k}^2 \big) \nonumber \\
	&\ + \sum_{i=0}^{q-1}\| \bfd_w^{i}\|_{\ast,\bfx_\ast^{i}}^2 + \tau \sum_{k=q}^n \big( \| \bfdd_u^k \|_{\ast,\bfx_\ast^k}^2 + \| \bfdd_u^k \|_{\ast,\bfx_\ast^k}^2 \big) ,
\end{align}
and where $C > 0$ is independent of $h$, $\tau$ and $n$, but depends exponentially on the final time $T$.
\end{prop}

In Section~\ref{Sec:Consistency} we will show that, under sufficient smoothness assumptions on the solution, the defects indeed satisfy the bounds
\begin{align}
D_h^n \leq C (h^{2k}+\tau^{2q}).
\end{align}
Hence, condition \eqref{eq:ConditionDefectsSatbility} is also satisfied under the step size restriction $\tau^q \leq C_0 h^k$. 
We note that the error functions $e_x^n, e_v^n \in S_h[\bfx_\ast^n]^3$ and $e_u^n, e_w^n \in S_h[\bfx_\ast^n]^4$ with nodal values $\bfe_x^n, \bfe_v^n$, $\bfe_u^n$ and $\bfe_w^n$, respectively, are then bounded by 
\begin{equation*}
\| e_x^n \|_{H^1(\Ga_h[\bfx_\ast^n])} + \| e_v^n \|_{H^1(\Ga_h[\bfx_\ast^n])} + \| e_u^n \|_{H^1(\Ga_h[\bfx_\ast^n])}  + \| e_w^n \|_{H^1(\Ga_h[\bfx_\ast^n])}\leq C(h^k+\tau^q),
\end{equation*}
for $n\tau \leq T$, provided the starting values are sufficiently accurate.

\begin{proof}
This proof adapts the semi-discrete stability proof of the numerical scheme \eqref{eq:matrix-form-X-v}--\eqref{eq:matrix--vector form} for Willmore flow in \cite[Section~5.4]{KovacsLiLubich2021} to the fully discrete setting. This is done by combining the techniques for evolving surfaces which lead to surface-dependent matrices developed in \cite[Section~10.3]{KovacsLiLubich2019} for the full discretization of mean curvature flow, and the adaptation of the structure of energy estimates developed in \cite[Section~6.4]{BullerjahnKovacs2024} for the full discretization of the Cahn--Hilliard equation with dynamic boundary conditions.

The proof is divided into tree parts. 
In Part (A) we obtain energy estimates for the surface PDEs by testing with the errors and their discrete time derivatives, exploiting the skew-symmetric structure of the error equations \eqref{eq:error eq - u} and \eqref{eq:error eq - w}, and using Dahlquist's $G$-stability theory (Lemma~\ref{Lemma:Dahlquist}) and the multiplier technique of Nevanlinna and Odeh (Lemma~\ref{Lemma:MultiplierTechnique}). The core idea of the proof is presented in the diagram in Figure~\ref{fig:EnergyEstimates}: (i) \& (ii) First an energy estimate for $\bfe_u^n$ is provided which comes with a critical term involving $\bfde_u^n$. (iii) \& (iv) Then in the second energy estimate we use the BDF time derivative of \eqref{eq:error eq - w}, which leads to a bound  for this critical term and for $\bfe_w^n$.
It is in this part, where for one of the estimates, \eqref{eq:Estimate Dr1 4}, we make use of the slightly modified system, Section~\ref{section:modified system} and Remark~\ref{rem:zHChange}, as well as the novel upper bound Lemma~\ref{Lemma:DahlquistUpper}.  	
In Part (B) we establish estimates for the errors in position and velocity by the equations \eqref{eq:error eq - x} and \eqref{eq:error eq - v}. 
Finally, Part (C) combines the final estimates of Parts (A) and (B) to show the stability bound \eqref{eq:StabilityBound}.

In the following $c$ is a generic constant, independent of $h$ and $\tau$, that may take different values on different occurrences. In contrast, constants with a subscript (such as $c_0$) will play a distinctive role in the proof. By $\rho>0$ we denote a small number, independent of $h$ and $\tau$, used in Young's inequalities, and hence we often incorporate multiplicative constants into those yet unchosen factors.

\begin{figure}[htbp]
	\begin{adjustbox}{width=\textwidth}
		\begin{tikzpicture}
			% the two main equations
			% (erroreq1) -- (erroreq2)
			\node at (2.5,5) (erroreq1) {$\tM^n \bfde_u^n - \tA^n \bfe_w^n = \bfr_1^n$};
			\node at (2.35,4.5) (erroreq2) {$\tM^n \bfe_w^n + \tA^n \bfe_u^n = \bfr_2^n$};
			%	\node at (2.5,4) (erroreq1) {$m_h(\dot e_h^u\t, \vphi_h^u) + a_h(e_h^w\t,\vphi_h^u) = \dotsc$};
			%	\node at (2.5,3.5) (erroreq2) {$m_h(e_h^w\t,\vphi_h^w) - a_h(e_h^u\t,\vphi_h^w) = \dotsc$};
			% the two sets of energy estimates
			
			% (energyest1) -- (energyest2)
			\node at (-1.95,-0.345) (energyest1) {\small $\displaystyle \|\bfe^n_u\|_{\bfK(\bfx_\ast^n)}^2 + \tau \! \sum \!\! ~\|\bfe_w^n\|_{\bfK(\bfx_\ast^n)}^2 \lesssim \! \underbrace{\tau \sum \!\! ~\|\dot{\bfe}_u^n\|_{\bfK(\bfx_\ast^n)}^2}_{\text{{\tiny critical term}}} +I_h^{q-1} + D_h^n$};
			\node at (9,-0.1) (energyest2) {\small $\displaystyle \|\bfe_w^n\|_{\bfK(\bfx_\ast^n)}^2 + \tau \! \sum \!\! ~\|\dot{\bfe}_u^n\|_{\bfK(\bfx_\ast^n)}^2 \lesssim I_h^{q-1}  +  D_h^n$};
			
			\pgfmathsetmacro{\yshftii}{-0.75}
			% the derivative and its arrows
			\node at (3,4.5+\yshftii) (dterroreq2) {$\tM^n \dot{\bfe}_w^n + \tA^n \dot{\bfe}_u^n = \bfR_2^n $};
			\draw[-stealth] (2.95,4.3) -- node[right, pos=0.45] {\tiny $\quad \partial^\tau_q$} (3.45,4.7+\yshftii);
			
			% the energy estimates (i)
			\pgfmathsetmacro{\xshfti}{-0.83}
			% arrows to the left, with adding or subtracting
			\draw[-stealth] (1.4+\xshfti,5.05) -- (-0.675+\xshfti,5.05) -- node[below, midway, sloped] {\footnotesize \hspace*{5 mm} test with $\bfe_u^n$} (-0.675+\xshfti,0.5);
			\draw[-stealth] (1.4+\xshfti,4.55) -- (-0.525+\xshfti,4.55) -- node[above, midway, sloped] {\footnotesize \hspace*{2.5 mm} test with $\bfe_w^n$} (-0.525+\xshfti,0.5);
			\draw (-0.6+\xshfti,0.4) node {\tiny $+$};
			
			\draw[-stealth] (1.4+\xshfti,4.95) --  (-2.075+\xshfti,4.95) -- node[below, midway, sloped] {\footnotesize \hspace*{7.5 mm} test with $\bfe_w^n$} (-2.075+\xshfti,0.5);
			\draw[-stealth] (1.4+\xshfti,4.45) --  (-1.925+\xshfti,4.45) -- node[above, midway, sloped] {\footnotesize \hspace*{2.5 mm} test with $\dot{\bfe}_u^n$}  (-1.925+\xshfti,0.5);
			\draw (-2+\xshfti,0.4) node {\tiny $-$};
			
			\draw (-2.1+\xshfti,5.35) node {\tiny (i)};
			\draw (-1.5,5.35) node {\tiny (ii)};
			
			% the energy estimates (ii)
			\pgfmathsetmacro{\xshftii}{2.18}
			\pgfmathsetmacro{\xshftiib}{-0.6}
			% arrows to the right, with adding or subtracting
			\draw[-stealth] (2.2+\xshftii,5.05) -- (7.925,5.05) -- node[below, midway, sloped] {\footnotesize \hspace*{15 mm} test with $\dot{\bfe}_u^n$} (7.925,0.5);
			\draw[-stealth] (5.3,4.55+\yshftii) -- (8.075,4.55+\yshftii) -- node[above, midway, sloped] {\footnotesize  test with $\bfe_w^n$} (8.075,0.5);
			\draw (8,0.4) node {\tiny $+$};
			
			\draw[-stealth] (2.2+\xshftii,4.95) -- (9.925+\xshftiib,4.95) -- node[below, midway, sloped] {\footnotesize \hspace*{15.5 mm} test with $\dot{\bfe}_w^n$} (9.925+\xshftiib,0.5);
			\draw[-stealth] (5.3,4.45+\yshftii) -- (10.075+\xshftiib,4.45+\yshftii) -- node[above, midway, sloped] {\footnotesize \hspace*{0 mm} test with $\dot{\bfe}_u^n $} (10.075+\xshftiib,0.5);
			\draw (10+\xshftiib,0.4) node {\tiny $-$};
			
			\draw (8.1+\xshftiib/2,5.35) node {\tiny (iii)};
			\draw (9.5+\xshftiib/2,5.35) node {\tiny (iv)};

			% stability and arrows
			\draw[draw, align=left] (4.15,-1.1) node {final estimate of Part (A)};
			\draw[draw, align=left] (2.5,-0.1) -- node[above, midway, sloped] {\footnotesize combining (i)--(iv)}  (6,-0.1);
			\draw[-stealth] (4.15,-0.1) -- (4.15,-0.65);
		\end{tikzpicture}
	\end{adjustbox}
	\caption{Sketch of the energy estimates for Part (A) of the fully discrete stability proof, with error equations \eqref{eq:error eq - u} and \eqref{eq:error eq - w}. (Note that, after discrete time differentiation, $\bfR_2^n$ contains more terms than only the discrete time derivative of $\bfr_2^n$.)} \label{fig:EnergyEstimates}
\end{figure}
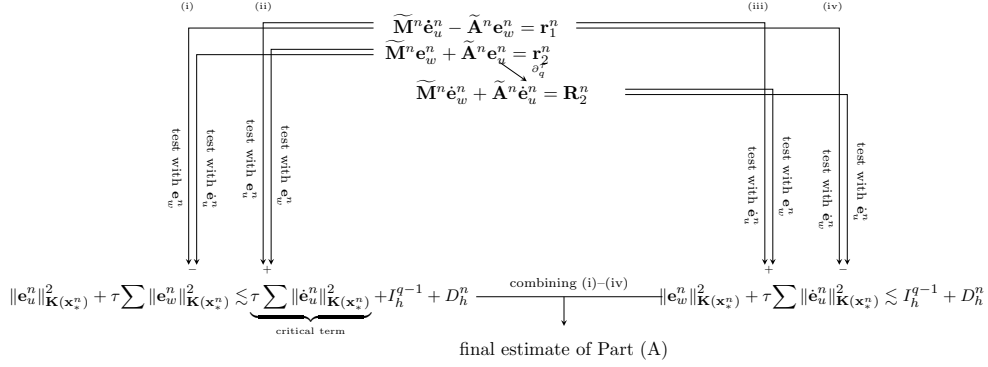

\emph{(Preparations):} Let $t_{\text{max}} \in (0,T]$ be the maximal time such that the following inequalities holds for $k\tau \leq (n-1)\tau \leq t_{\text{max}}$:
\begin{equation} \label{eq:W1InftyEstimates}
	\begin{aligned}
		\|e_x^k \|_{W^{1,\infty}(\Ga_h[\bfx_\ast^k])} \leq&\ h^{(\kappa-1)/2}, \\
		\|e_v^k \|_{W^{1,\infty}(\Ga_h[\bfx_\ast^k])} \leq&\ h^{(\kappa-1)/2}, \\
		\|e_u^k \|_{W^{1,\infty}(\Ga_h[\bfx_\ast^k])} \leq&\ h^{(\kappa-1)/2}, \\
		\|e_w^k \|_{W^{1,\infty}(\Ga_h[\bfx_\ast^k])} \leq&\ h^{(\kappa-1)/2}, 
	\end{aligned}
\end{equation}
Note that by an inverse inequality, assumption \eqref{eq:AssInitialValues} and $v_h^n = \widetilde I_h(V_h^n \nu_h^n)$, these estimates are satisfied at least for $(q-1)\tau \leq t_{\text{max}}$. We first prove the stated bounds for $(n-1)\tau \leq t_{\text{max}}$ and at the end of the proof argue by a bootstrapping argument that in fact $t_{\text{max}}$ coincides with $T$.

The estimate on the position errors $\bfe_x^k$ for $k\tau \leq (n-1)\tau \leq t_{\text{max}}$ in \eqref{eq:W1InftyEstimates}, the $W^{1,\infty}$ bound on $v_h^n$ and the fact that the extrapolation is a bounded operator, immediately imply that, by the results of Section~\ref{section:aux} and the regularity of the exact solution, the norms corresponding to the positional vectors $\bftx^j$, $\bftx_\ast^j$ and $\bfx_\ast^i$, for $q\leq j \leq n$ and $0\leq i \leq n$, are $h$-uniformly equivalent (for sufficiently small $h\leq h_0$ and $\tau \leq \tau_0$). Additionally we are in a position to use the matrix difference bounds \eqref{matrix difference bounds} and \eqref{matrix difference bounds e_x}. 

In particular, for $\widetilde V_h^n \in S_h[\widetilde \bfx^n]$, the finite element function on $\Ga_h[\bftx^n]$ with nodal vector $\widetilde \bfV^n = \partial^\tau \bftx^n$, we have by the exact same argument as in \cite[equation~(10.10)]{KovacsLiLubich2019}:
\begin{align}
	\| \widetilde V_h^n\|_{W^{1,\infty}(\Ga_h[\bftx^n])} \leq K,
\end{align}
with a constant $K$ independent of $h$ and $\tau$. This results, by $\bftx^n-\bftx^{n-1}=\tau \widetilde \bfV^n$ and using \eqref{matrix difference bounds}, that
\begin{equation}
	\label{eq:oTau bound}
	\begin{aligned}
		\bfw^\top (\tM^n-\tM^{n-1}) \bfz \leq &\ c \, \tau \, \|\bfw\|_{\bfM^n} \|\bfz\|_{\bfM^n} , \\
		\bfw^\top (\tA^n-\tA^{n-1}) \bfz \leq &\ c \, \tau \, \|\bfw\|_{\bfA^n} \|\bfz\|_{\bfA^n} ,
	\end{aligned}
\end{equation}
for arbitrary vectors $\bfw$ and $\bfz$, and $\tau \leq \tau_0$ sufficiently small. 

Next we will obtain a similar estimate for the non-linear part involving $\bfF$, using local Lipschitz continuity of $F$, denoting the non-linearity inside the integral, cf. \eqref{eq:alpha mass matrix - F_1}, \eqref{eq:alpha mass matrix - F_2} and \eqref{eq:bfF definition}, and \eqref{eq:W1InftyEstimates} for the terms with $\bftu$. Following the semi-discrete arguments in \cite[equation~(5.35)]{KovacsLiLubich2021}, adapted to the fully discrete setting by a combination of the arguments in \cite[equation~(6.28)]{BullerjahnKovacs2024} and \cite[equation~(10.27)]{KovacsLiLubich2019}, we first establish, for $1\leq l \leq n$, and by bounding $z_h$ in $L^\infty$:

\begin{align}
	\bfw^\top \partial^\tau \big(\bfF(\bftx^l,\bftu^l) - \bfF(\bftx_\ast^l,\bftu_\ast^l)\big) \bfz \leq c \, \| z_h\|_{L^\infty} \| \bfw \|_{\bfM^l} \big(&\| \partial^\tau \bfte_u^l \|_{\bfK^{l}} + \| \bfte_u^l \|_{\bfK^{l}} + \| \bfte_u^{l-1} \|_{\bfK^{l-1}} \nonumber \\
	&\ + \|\partial^\tau \bfte_x^l \|_{\bfM^l} + \| \bfte_x^l \|_{\bfK^l} + \| \bfte_x^{l-1} \|_{\bfM^{l-1}}\big). \label{eq:dG Est inText}
\end{align} 

We define $\bftu_\theta^l:= \bftu_\ast^l + \theta \bfte_u^l$ for $\theta \in [0,1]$. Then we have 
\begin{align}
	\bfw^\top \partial^\tau \big(\bfF(\bftx^l,\bftu^l) - \bfF(\bftx_\ast^l,\bftu_\ast^l)\big) \bfz =&\ \int_0^1 \frac{\dif}{\dif \theta} \frac1\tau (\bfw)^\top (\tM^l - \tM^{l-1}) F(\bftu_\theta^l,\nabla_{\Ga_h} \bftu_\theta^l) \bfz \dif \theta \nonumber \\
	& + \int_0^1 \frac{\dif}{\dif\theta} \bfw^\top \tM^{l-1} \partial^\tau F(\bftu_\theta^l,\nabla_{\Ga_h} \bftu_\theta^l) \bfz\dif \theta \nonumber \\
	& + \bfw^\top \partial^\tau \big(\tM^l - \tM_\ast^l\big) F(\bftu_\ast^l,\nabla_{\Ga_h} \bftu_\ast^l) \bfz\nonumber \\
	& + \bfw^\top \big(\tM^l - \tM_\ast^l\big) \partial^\tau F(\bftu_\ast^l,\nabla_{\Ga_h} \bftu_\ast^l) \bfz\nonumber \\
	=:&\ (I) + (II) + (III) + (IV). \label{eq:dg Est tested inText}
\end{align}

We estimate the first term using estimate \eqref{eq:oTau bound} and local Lipschitz continuity of $g$ and \eqref{eq:W1InftyEstimates}, as
\begin{align}
	(I) =&\  \int_0^1 \frac1\tau \bfw^\top (\tM^l - \tM^{l-1}) \Big(\partial_1 F(\bftu_\theta^l,\nabla_{\Ga_h} \bftu_\theta^l) \bfte_u^l + \partial_2 F(\bftu_\theta^l,\nabla_{\Ga_h} \bftu_\theta^l) \nabla_{\Ga_h} \bfte_u^l\Big) \bfz\dif \theta \nonumber \\
	\leq&\ c \, \| z_h\|_{L^\infty} \| \bfw \|_{\bfM^l} \| \bfte_u^l \|_{\bfK^l}. \label{eq:dg Est 1 inText}
\end{align}

The second term is estimated by local Lipschitz continuity of $g$ and its derivatives and \eqref{eq:W1InftyEstimates}, to
\begin{align}
	(II) =&\ \int_0^1 \bfw^\top \tM^{l-1} \partial^\tau \Big(\partial_1 F(\bftu_\theta^l,\nabla_{\Ga_h} \bftu_\theta^l) \bfte_u^l + \partial_2 F(\bftu_\theta^l,\nabla_{\Ga_h} \bftu_\theta^l) \nabla_{\Ga_h} \bfte_u^l\Big) \bfz \dif \theta \nonumber \\
	=&\ \int_0^1 \bfw^\top \tM^{l-1} \Big(\partial_1 F(\bftu_\theta^{l-1},\nabla_{\Ga_h} \bftu_\theta^{l-1}) \partial^\tau \bfte_u^l + \partial^\tau \partial_1 F(\bftu_\theta^l,\nabla_{\Ga_h} \bftu_\theta^l)  \bfte_u^l\Big) \bfz \dif \theta \nonumber \\
	& + \int_0^1 \bfw^\top \tM^{l-1} \Big(\partial_2 F(\bftu_\theta^l,\nabla_{\Ga_h} \bftu_\theta^l) \nabla_{\Ga_h}\partial^\tau \bfte_u^l + \partial^\tau \partial_2 F(\bftu_\theta^l,\nabla_{\Ga_h} \bftu_\theta^l)  \nabla_{\Ga_h} \bfte_u^{l-1}\Big) \bfz \dif \theta\nonumber \\
	\leq&\ c \, \| z_h\|_{L^\infty} \| \bfw \|_{\bfM^{l-1}} \big(\| \partial^\tau \bfte_u^l \|_{\bfK^{l-1}} + \| \bfte_u^l \|_{\bfK^{l-1}} + + \| \bfte_u^{l-1} \|_{\bfK^{l-1}} \big). \label{eq:dg Est 2 inText}
\end{align}

The arguments from \cite[equation~(10.21)]{KovacsLiLubich2019}, for $\tM$ instead of $\tA$, yield the estimate
\begin{align}
	(III) \leq&\  c \, \| z_h\|_{L^\infty} \| \bfw \|_{\bfM^l} \big(\|\partial^\tau \bfte_x^l \|_{\bfM^l} + \| \bfte_x^l \|_{\bfM^l} + \| \bfte_x^{l-1} \|_{\bfM^{l-1}}\big). \label{eq:dg Est 3 inText}
\end{align}

Finally, the fourth term is estimated, using \eqref{matrix difference bounds e_x} and the regularity of the exact solution, by
\begin{equation}
	(IV) \leq c \, \| z_h\|_{L^\infty} \| \bfw \|_{\bfM^l} \| \bfte_x^l \|_{\bfA^l}. \label{eq:dg Est 4 inText}
\end{equation}

Inserting the estimates \eqref{eq:dg Est 1 inText}, \eqref{eq:dg Est 2 inText}, \eqref{eq:dg Est 3 inText}, and \eqref{eq:dg Est 4 inText} into \eqref{eq:dg Est tested inText}, and using the norm equivalence \eqref{norm-equiv}, we obtain the bound \eqref{eq:dG Est inText}.

Furthermore, we obtain the following analogous result to \eqref{eq:oTau bound}, using the regularity of the exact solution, and changing to the $W^{1,\infty}$ norm in $\bfe_x$ and $\bfe_u$, if it is possible by \eqref{eq:W1InftyEstimates}:
\begin{align}
	\bfw^\top (\bfF(\bftx^n,\bftu^{n})-\bfF(\bftx^{n-1},\bftu^{n-1})) \bfz =&\ \tau \, \bfw^\top \partial^\tau(\bfF(\bftx^n,\bftu^{n}) - \bfF(\bftx_\ast^n,\bftu_\ast^{n})) \bfz \nonumber \\
	&\ + \tau \, \bfw^\top \partial^\tau\bfF(\bftx_\ast^n,\bftu_\ast^{n}) \bfz \nonumber \\
	\leq&\ c \, \tau \| \bfw \|_{\bfM^n} \big(\| \partial^\tau \bfte_x^n \|_{\bfK^n} + \|  \partial^\tau \bfte_u^n \|_{\bfK^n} + \|\bfz \|_{\bfM^n} \big). \label{eq:oTau bound F} 
\end{align}
Note that we still need a $L^\infty$-bound for $z_h$, which is possible in all applications, because of \eqref{eq:W1InftyEstimates}.

Finally, we have the following estimate for the non-linear part, similar to \eqref{matrix difference bounds}, where Lipschitz continuity of $F$ together with \eqref{eq:W1InftyEstimates} and \eqref{matrix difference bounds e_x}, yield
\begin{align}
	\bfw^\top (\bfF(\bftx^n,\bftu^{n})-\bfF(\bftx_\ast^{n},\bftu_\ast^{n})) \bfz =&\ \bfw^\top \tM^n \Big[\big(F(\bftu^n,\nabla_{\Ga_h} \bftu^{n})-F(\bftu_\ast^{n},\nabla_{\Ga_h} \bftu_\ast^{n})\big) \bfz\Big] \nonumber \\
	&\ + \bfw^\top (\tM^n - \tM_\ast^n)\Big[F(\bftu_\ast^n,\nabla_{\Ga_h}\bftu_\ast^{n}) \bfz\Big] \nonumber \\
	\leq&\ c \, \| \bfw \|_{\bfM^n} \big( \| \bfte_u^n \|_{\bfK^n} + \| \bfte_x^n \|_{\bfA^n} \big).\label{eq:non-linear matrix difference}
\end{align}

\emph{(A) Estimates for the surface PDEs:} 

(A.1) We start by the energy estimates (i) \& (ii) in Figure~\ref{fig:EnergyEstimates}. Let $q+1\leq k \leq n$, such that $(n-1)\tau \leq t_{\text{max}}$, test $\eqref{eq:error eq - w}^k$ (where the superscript $k$ denotes in which time step we consider the equation) with $\bfde_u^k + \bfe_w^k$, and add $\eqref{eq:error eq - u}^k$ tested with $\bfe_u^k - \bfe_w^k$, to obtain
\begin{align}
	(\bfde_u^k)^\top \tK^k\bfe_u^k + \| \bfe_w^k \|_{\bfK^k}^2 =&\ (\bfde_u^k)^\top \bfr_2^k + (\bfe_w^k)^\top \bfr_2^k + (\bfe_u^k)^\top \bfr_1^k - (\bfe_w^k)^\top \bfr_1^k \label{eq:EnergyITested}
\end{align}   
The left-hand side is estimated by the combination of Lemma~\ref{Lemma:Dahlquist} and Lemma~\ref{Lemma:MultiplierTechnique} (note here that in the present case $q=1,2$, the multiplier $\eta_q$ vanishes), with the symmetric positive definite matrix $G$ and $\bfE_u^k = (\bfe_u^k,\dots,\bfe_u^{k-q+1})$, as
\begin{align*}
	\frac1\tau \Big(|\bfE_u^k|_{G,\bfK^k}^2 - |\bfE_u^{k-1}|_{G,\bfK^k}^2\Big) \leq (\bfde_u^k)^\top \tK^k\bfe_u^k,
\end{align*}
and summing over $q+1 \leq k \leq n$ and multiplying by $\tau$, see the arguments following \cite[equation~(10.29)]{KovacsLiLubich2019}, and using \eqref{eq:GnormEst}, yields the bound
\begin{align}
	\tau \sum_{k=q+1}^n (\bfde_u^k)^\top \tK^k\bfe_u^k \geq \frac12 \lambda_0 \sum_{j=0}^{q-1} \| \bfe_u^{n-j} \|_{\bfK^{n-j}}^2 - c \sum_{j=1}^q \| \bfe_u^i \|_{\bfK^i}^2 - c \, \tau \sum_{k=q+1}^{n-1} \| \bfe_u^k \|_{\bfK^k}^2. \label{eq:EnergyILHS}
\end{align}
While the terms on the right-hand side are estimated separately: We consider $\bfr_1^k$ multiplied with the place-holder $\bfe^k$, and summing over $q+1 \leq k \leq n$ and multiplying by $\tau$, which yields
\begin{align*}
	\tau \sum_{k=q+1}^n (\bfe^k)^\top \bfr_1^k =&\ \tau \sum_{k=q+1}^n \Big( -(\bfe^k)^\top(\tM^k - \tM_\ast^k) \bfdu_\ast^k + (\bfe^k)^\top(\tA^k - \tA_\ast^k) \bfw_\ast^k - (\bfe^k)^\top \bfF(\bftx^k,\bftu^k) \bfte_w^k \nonumber \\
	&\ - (\bfe^k)^\top (\bfF(\bftx^k,\bftu^k)-\bfF(\bftx_\ast^k,\bftu_\ast^k))\bftw_\ast^k + (\bfe^k)^\top (\bff(\bftx^k,\bftu^k)-\bff(\bftx_\ast^k,\bftu_\ast^k)) \nonumber \\
	&\ - (\bfe^k)^\top \tM_\ast^k \bfd_u^k \Big)\nonumber \\
	&\ =: (I) + (II) + (III) + (IV) + (V) + (VI).
\end{align*}
We follow the lines of the semi-discrete proof \cite[equations~(5.23)--(5.26)]{KovacsLiLubich2021} adapted to the fully discrete case by the techniques developed in \cite[equations~(10.18),(10.27)--(10.28)]{KovacsLiLubich2019}. We estimate the terms $(I)$ and $(II)$ by \eqref{matrix difference bounds e_x}, with the help of \eqref{eq:W1InftyEstimates}, $(III)$ by \eqref{eq:W1InftyEstimates}, Cauchy--Schwarz and Young's inequality, and $(VI)$ by the defect norm and Young's inequality. In order to bound the terms $(IV)$ and $(V)$ we use local Lipschitz continuity and the bound \eqref{eq:non-linear matrix difference}, exactly as in \cite[equations~(10.27)]{KovacsLiLubich2019}, to finally obtain
\begin{align}
	\tau \sum_{k=q+1}^n (\bfe^k)^\top \bfr_1^k \leq&\ \rho \tau \sum_{k=q+1}^n \| \bfe^k \|_{\bfK^k}^2 + c \ \epsilon_h^n + c \, D_h^n, \label{eq:EstRHS1}
\end{align}
where we collect terms to be estimated later by Gr\"onwall's inequality and the equations \eqref{eq:error eq - x} and \eqref{eq:error eq - v}, the initial values and terms for $k=q$ (already estimated by \eqref{eq:qest}), in
\begin{align}
	\epsilon_h^n:=& \tau \sum_{k=q+1}^n \|\bfe_w^k \|_{\bfK^k}^2 +  \tau \sum_{k=q+1}^{n} \|\bfe_u^k \|_{\bfK^k}^2 + \tau \sum_{k=q}^n \|\bfe_x^k \|_{\bfK^k}^2 + \tau \sum_{k=q}^n \| \bfe_v^k \|_{\bfK^k}^2 + I_h^{q-1} + D_h^q. \label{eq:DefEpsilon}
\end{align}

In the same way, with even simpler terms, we obtain for the other non-linear term $\bfr_2^k$ the exact same estimate
\begin{align}
	\tau \sum_{k=q+1}^n (\bfe^k)^\top \bfr_2^k \leq&\ \rho \tau \sum_{k=q+1}^n \| \bfe^k \|_{\bfK^k}^2 + c \, \epsilon_h^n + c \, D_h^n. \label{eq:EstRHS2}
\end{align}

%
%\begin{align}
%	(\bfde_u^k)^\top(\bfg(\bftx^k,\bftu^k) - \bfg(\bftx_\ast^k,\bftu_\ast^k)) \leq&\ \rho \| \bfde_u^k \|_{\bfM^k}^2 + c \, \| \bfte_x^k \|_{\bfK^k}^2 + c \, \| \bfte_u^k \|_{\bfK^k}^2, \label{eq:EnergyIRHS3}
%\end{align}
%by \eqref{eq:alpha mass matrix - F_1} and the bounds \eqref{eq:W1InftyEstimates},		
%\begin{align}
%	(\bfe_w^k)^\top \bfF(\bftx^k,\bftu^k)\bfe_w^k \leq&\ c \, \|\bfe_w^k \|_{\bfM^k}^2 \label{eq:EnergyIRHS4}
%\end{align}
%and finally the terms involving the defects by the dual norm, Young's inequality and the norm equivalence, as
%\begin{align}
%	(\bfde_u^k)^\top \bfM_\ast^k \bfd_w^k + (\bfde_u^k)^\top \bfvartheta^k \leq&\ \rho \| \bfde_u^k \|_{\bfK^k}^2 + c \, \| \bfd_w^k \|_{\ast,\bftx_\ast^k}^2 + c \, \| \bfd_w^{q-1} \|_{\ast,\bftx_\ast^{q-1}}^2. \label{eq:EnergyIRHS5}
%\end{align}

Altogether, the combination of \eqref{eq:EnergyITested} with the above bounds \eqref{eq:EnergyILHS}, \eqref{eq:EstRHS1} and \eqref{eq:EstRHS2}, the estimate for the case $n=q$ in \eqref{eq:qest}, and the norm equivalence yields the final energy estimate for the branch on the left-hand side of Figure~\ref{fig:EnergyEstimates}:
\begin{align}
	\sum_{k=n-q+1}^n \| \bfe_u^k \|_{\bfK^k}^2 + \tau \sum_{k=q+1}^n \| \bfe_w^k \|_{\bfK^k}^2 \leq&\ \rho \tau \sum_{k=q+1}^n \|\bfde_u^k \|_{\bfK^k}^2 + c \, \epsilon_h^n + c \, D_h^n  \label{eq:EnergyA.1}
\end{align}

\textit{(The special case $n=q$)} For $n=q$ it remains to show an estimate analogous to the final estimate in part (A). By slightly modifying the argument for the general case we obtain
\begin{align}
	\label{eq:qest}
	\| \bfe_u^q \|_{\bfK^q}^2 + \| \bfe_w^q \|_{\bfK^q}^2 + \tau \| \bfde_u^q \|_{\bfK^q}^2  \leq&\  c \, \tau \|\bfe_x^q \|_{\bfK^q}^2 + c\tau \| \bfe_v^q \|_{\bfK^q}^2 + c \, I_h^{q-1} + c \, D_h^q .
\end{align}	
The difference to the general case is that there is no equation for $n-1=q-1$. We solve this problem by adding the missing terms on both sides, so that we are able use Dahlquist's $G$-stability result, Lemma~\ref{Lemma:Dahlquist}, and the multiplier technique of Nevanlinna and Odeh, Lemma~\ref{Lemma:MultiplierTechnique}, these terms involve only initial values and therefore we estimate them directly.

(A.2) We will now perform energy estimates (iii) sketched in the branch on the right-hand side of Figure~\ref{fig:EnergyEstimates}. Following the semi-discrete proof in \cite[Section~5.4]{KovacsLiLubich2021} we differentiate the second equation, but now discretely in time. This is done in the same way as in \cite[Section~6.4]{BullerjahnKovacs2024}: Let again $q+1\leq k \leq n$, with $(n-1)\tau \leq t_{\text{max}}$. 

For $k\geq 2q$, we consider the linear combination of \eqref{eq:error eq - w}, for $k-i=k-q,\dots,k$, weighted by $\delta_i/\tau$, which gives
\begin{align*}
	\tM^k \bfde_w^k - \tA^k \bfde_u^k = \bfR_2^k, \qquad k\geq 2q,
\end{align*}
with 
\begin{align*}
	\bfR_2^k = \frac1\tau \sum_{i=0}^q \delta_i (\tM^k - \tM^{k-i})\bfe_w^{k-i} + \frac1\tau \sum_{i=0}^q \delta_i (\tA^k - \tA^{k-i})\bfe_u^{k-i} + \frac1\tau \sum_{i=0}^q \delta_i \bfr_2^{k-i} .
\end{align*}

For $k<2q$ we have to use a different equation, since then in the discrete derivative the initial values are starting to appear, and for the initial values we do not have the equation \eqref{eq:error eq - w}. So instead we have to add the terms involving the initial values on both sides. For convenience of notation we define
\begin{align*}
	\bfr_2^i := \tM^i \bfe_w^i + \tA^i \bfe_u^i, \forquad i=0,\dots,q-1.
\end{align*} 
This allows us to write the differentiated second equation the following compact form, for $q+1\leq k <2q$: 
\begin{align}
	\tM^k \bfde_w^k - \tA^k \bfde_u^k = \bfR_2^k \label{eq:error Dw}
\end{align}

In order to obtain the estimate (iii) we now test $\eqref{eq:error eq - u}^k$ with $\bfde_u^k$ and add $\eqref{eq:error Dw}^k$ tested with $\bfe_w^k$, which yields:
\begin{align*}
	(\bfe_w^k)^\top \tM^k \bfde_w^k + \| \bfde_u^k \|_{\bfM^k}^2 = (\bfde_u^k)^\top\bfr_1^k + (\bfe_w^k)^\top \bfR_2^k.
\end{align*}

For the left-hand side, we again use a combination of Lemma~\ref{Lemma:Dahlquist} and Lemma~\ref{Lemma:MultiplierTechnique}, to obtain as in \eqref{eq:EnergyILHS}, after summing over $q+1 \leq k \leq n$ and multiplying by $\tau$,
\begin{align*}
	\sum_{k=n-q+1}^n \| \bfe_w^k \|_{\bfM^k}^2 + \tau \sum_{k=q+1}^n \| \bfde_u^k \|_{\bfM^k}^2 \leq&\ c \, \tau \sum_{k=q+1}^n \Big( (\bfde_u^k)^\top\bfr_1^k + (\bfe_w^k)^\top \bfR_2^k\Big) + c \, \epsilon_h^n,
\end{align*}
and with \eqref{eq:EstRHS2}, we arrive at
\begin{align}
	\sum_{k=n-q+1}^n \| \bfe_w^k \|_{\bfM^k}^2 + \tau \sum_{k=q+1}^n \| \bfde_u^k \|_{\bfM^k}^2 \leq&\ c \, \tau \sum_{k=q+1}^n (\bfe_w^k)^\top \bfR_2^k + \rho \tau \sum_{k=q+1}^n \|\bfde_u^k\|_{\bfK^k}^2 +  c \, \epsilon_h^n + c \, D_h^n. \label{eq:EnergyIII before R2 estimate}
\end{align}

The only term left to be estimated is the differentiated right-hand side $\bfR_2^k$, here we again use a place holder $\bfe^k$ to estimate in general:
\begin{align}
	\tau \sum_{k=q+1}^n (\bfe^k)^\top \bfR_2^k =&\ \tau \sum_{k=q+1}^n \sum_{i=0}^q \frac1\tau \delta_i \Big[ -(\bfe^k)^\top(\tM^k - \tM^{k-i}) \bfe_w^{k-i} - (\bfe^k)^\top(\tA^k - \tA^{k-i}) \bfe_u^{k-i} \nonumber \\
	&\ \hphantom{\tau \sum_{k=q+1}^n \sum_{i=0}^q \frac1\tau \delta_i }- (\bfe^k)^\top(\tM^{k-i} - \tM_\ast^{k-i}) \bfw_\ast^{k-i} - (\bfe^k)^\top(\tA^{k-i} - \tA_\ast^{k-i}) \bfu_\ast^{k-i}  \nonumber \\
	&\ \hphantom{\tau \sum_{k=q+1}^n \sum_{i=0}^q \frac1\tau \delta_i }+ (\bfe^k)^\top (\bfg(\bftx^{k-i},\bftu^{k-i}) - \bfg(\bftx_\ast^{k-i},\bftu_\ast^{k-i})) \nonumber \\
	&\ \hphantom{\tau \sum_{k=q+1}^n \sum_{i=0}^q \frac1\tau \delta_i }- (\bfe^k)^\top \tM_\ast^{k-i} \bfd_w^{k-i} + (\bfe^k)^\top \bfvartheta^{k-i}  \Big] \nonumber \\
	=:&\ (I) + (II) + (III) + (IV). \label{eq:Estimate R2 tested}
\end{align}

We estimate the first term $(I)$ by \eqref{eq:oTau bound} and Young's inequality: 
\begin{align}
	\tau \sum_{k=q+1}^n \sum_{i=0}^q \frac1\tau \delta_i (\bfe^k)^\top(\tM^k - \tM^{k-i}) \bfe_w^{k-i} + (\bfe^k)^\top(\tA^k - \tA^{k-i}) \bfe_u^{k-i} \leq \rho \tau \sum_{k=q+1}^n \| \bfe^k \|_{\bfK^k}^2 + c \, \epsilon_h^n. \label{eq:Estimate R2 1}
\end{align}

By the factorization $\delta(\zeta) = (1-\zeta) \sigma(\zeta)$ of the generating polynomial $\delta$ of the backwards difference, cf. \eqref{eq:BDF generating function}, see \cite[equation~(10.12)--(10.14)]{KovacsLiLubich2019}, and the estimates \eqref{eq:W1InftyEstimates}, \cite[equation~(10.21)]{KovacsLiLubich2019} and \eqref{eq:oTau bound}, we obtain for the first term in $(II)$
\begin{align}
	\tau \sum_{k=q+1}^n \sum_{i=0}^q \frac1\tau \delta_i (\bfe^k)^\top(\tM^{k-i} - \tM_\ast^{k-i}) \bfw_\ast^{k-i}  =&\ \tau \sum_{k=q+1}^n \sum_{i=0}^{q-1} \sigma_i (\bfe^k)^\top\partial^\tau(\tM^{k-i} - \tM_\ast^{k-i}) \bfw_\ast^{k-i}  \nonumber \\
	&\ \hphantom{\sum_{k=q+1}^n\sum_{i=0}^q }+  \sigma_i (\bfe^k)^\top(\tM^{k-i} - \tM_\ast^{k-i}) \partial^\tau \bfw_\ast^{k-i} \nonumber \\
	\leq&\ \rho \tau \sum_{k=q+1}^n \| \bfe^k \|_{\bfK^k}^2 + c \, \epsilon_h^n, \label{eq:Estimate R2 2}
\end{align}
where we have additionally used a slight modification of \cite[equation~(10.32)]{KovacsLiLubich2019}, obtained from the estimates leading up to the cited one:
\begin{align}
	\tau \sum_{k=q}^n \|\partial^\tau \bfe_x^k \|_{\bfK^n}^2 \leq c \, \tau \sum_{k=q}^n \|\bfde_x^n\|_{\bfK^k}^2 + c \, \tau \sum_{i=1}^{q-1}  \|\partial^\tau \bfe_x^i \|_{\bfK^k}^2 \leq c \, \epsilon_h^n. \label{eq:Partial by BDFq}
\end{align}
The second term in $(II)$ can be estimated in exactly the same way.

The term $(IV)$ is estimated by reformulating the backward difference and via the dual norms and again \eqref{eq:oTau bound}, as 
\begin{align}
	\tau \sum_{k=q+1}^n \sum_{i=0}^q \frac1\tau \delta_i (\bfe^k)^\top \tM_\ast^{k-i} \bfd_w^{k-i} + (\bfe^k)^\top \bfvartheta^{k-i} \leq \rho \tau \sum_{k=q+1}^n \| \bfe^k \|_{\bfK^k}^2 + c \, \epsilon_h^n + c \, D_h^n \label{eq:Estimate R2 3}
\end{align}
note that the backward difference of the defect term $\bfd_w$ enters in $D_h^n$, and that $\partial^\tau \bfvartheta^k=0$ for $k\geq q$, see \eqref{eq:DefVartheta}.

In order to estimate the remaining non-linear term $(III)$ we follow the semi-discrete arguments in \cite[equation~(5.35)]{KovacsLiLubich2021}, adapted to the fully discrete setting as in \eqref{eq:dG Est inText}, with again the factorization $\delta(\zeta) = (1-\zeta) \sigma(\zeta)$ of the generating polynomial $\delta$ of the backward difference. This leads to the estimate
\begin{align}
	\tau \sum_{k=q+1}^n \sum_{i=0}^q \frac1\tau \delta_i (\bfe^k)^\top (\bfg(\bftx^{k-i},\bftu^{k-i}) - \bfg(\bftx_\ast^{k-i},\bftu_\ast^{k-i})) \leq&\  c \, \tau \sum_{k=q+1}^n \|\bfe^k\|_{\bfM^k} \| \bfde_u^k\|_{\bfK^k} \nonumber \\
	&\ + \rho \tau \sum_{k=q+1}^n \| \bfe^k \|_{\bfK^k}^2 + c \, \epsilon_h^n. \label{eq:Estimate R2 4}
\end{align}

Inserting now the estimates \eqref{eq:Estimate R2 1}, \eqref{eq:Estimate R2 2}, \eqref{eq:Estimate R2 3}, and \eqref{eq:Estimate R2 4} into \eqref{eq:Estimate R2 tested}, we arrive at
\begin{align}
	\tau \sum_{k=q+1}^n (\bfe^k)^\top \bfR_2^k \leq c \, \tau \sum_{k=q+1}^n \|\bfe^k\|_{\bfM^k} \| \bfde_u^k\|_{\bfK^k} + \rho \tau \sum_{k=q+1}^n \| \bfe^k \|_{\bfK^k}^2 + c \, \epsilon_h^n + c \, D_h^n. \label{eq:Estimate R2}
\end{align}

Using this estimate and Young's inequality, together with \eqref{eq:EnergyIII before R2 estimate}, we arrive at the third energy estimate
\begin{align}
	\sum_{k=n-q+1}^n \| \bfe_w^k \|_{\bfM^k}^2 + \tau \sum_{k=q+1}^n \| \bfde_u^k \|_{\bfM^k}^2 \leq&\ \rho \tau \sum_{k=q+1}^n \|\bfde_u^k\|_{\bfK^k}^2 +  c \, \epsilon_h^n + c \, D_h^n. \label{eq:EnergyIII}
\end{align}

For the energy estimate (iv) of Figure~\ref{fig:EnergyEstimates}, we test $\eqref{eq:error Dw}^k$ with $\bfde_u^k$, and subtract $\eqref{eq:error eq - u}^k$ tested with $\bfde_w^k$, sum over $q+1 \leq k \leq n$ and multiply by $\tau$, use again Lemma~\ref{Lemma:Dahlquist}, Lemma~\ref{Lemma:MultiplierTechnique} and then we arrive as in \eqref{eq:EnergyILHS}, at
\begin{align*}
	\sum_{k=n-q+1}^n \| \bfe_w^k \|_{\bfA^k}^2 + \tau \sum_{k=q+1}^n \| \bfde_u^k \|_{\bfA^k}^2 \leq&\ -\tau \sum_{k=q+1}^n (\bfde_w^k)^\top \bfr_1^k + \tau \sum_{k=q+1}^n (\bfde_u^k)^\top \bfR_2^k + c \, \epsilon_h^n.
\end{align*}
Using here estimate \eqref{eq:Estimate R2} and Young's inequality, we obtain
\begin{align}
	\sum_{k=n-q+1}^n \| \bfe_w^k \|_{\bfA^k}^2 + \tau \sum_{k=q+1}^n \| \bfde_u^k \|_{\bfA^k}^2 \leq&\ -\tau \sum_{k=q+1}^n (\bfde_w^k)^\top \bfr_1^k + c_1 \tau \sum_{k=q+1}^n \|\bfde_u^k\|_{\bfM^k}^2 \nonumber \\
	&\ + \rho \tau \sum_{k=q+1}^n \|\bfde_u^k\|_{\bfK^k}^2 + c \, \epsilon_h^n + c \, D_h^n. \label{eq:Energy4 tested}
\end{align}

To estimate the remaining term on the right-hand side, involving $\bfde_w^k$, we aim to transfer the derivative onto $\bfr_1^k$ by a discrete product rule, compare to the energy estimates (IV) in \cite{KovacsLiLubich2021} and \cite{BullerjahnKovacs2024}. We have
\begin{align}
	-\tau \sum_{k=q+1}^n (\bfde_w^k)^\top \bfr_1^k =&\ \tau \sum_{k=q+1}^n (\bfde_w^k)^\top (\tM^k-\tM_\ast^k) \bfdu_\ast^k - (\bfde_w^k)^\top (\tA^k-\tA_\ast^k) \bfw_\ast^k \nonumber \\
	&\ \hphantom{\sum_{k=q+1}^n}+ (\bfde_w^k)^\top \bfF(\bftx^k,\bftu^k) \bfte_w^k + (\bfde_w^k)^\top (\bfF(\bftx^k,\bftu^k)-\bfF(\bftx_\ast^k,\bftu_\ast^k)) \bftw_\ast^k \nonumber \\
	&\ \hphantom{\sum_{k=q+1}^n}-  (\bfde_w^k)^\top (\bff(\bftx^k,\bftu^k)-\bff(\bftx_\ast^k,\bftu_\ast^k)) + (\bfde_w^k)^\top \tM_\ast^k \bfd_u^k \nonumber \\
	=:&\ (I) + (II) + (III) + (IV) + (V) + (VI). \label{eq:Estimate Dr1 tested}
\end{align}

The terms $(I)$ and $(II)$ are estimated as in \cite[Proposition~10.1, Part(A.iv)]{KovacsLiLubich2019}, by
\begin{align}
	%\tau \sum_{k=q+1}^n (\bfde_w^k)^\top (\tM^k-\tM_\ast^k) \bfdu_\ast^k - (\bfde_w^k)^\top (\tA^k-\tA_\ast^k) \bfw_\ast^k 
	(I) + (II) \leq \rho \sum_{k=n-q+1}^n \| \bfe_w^k\|_{\bfK^k}^2 + c \sum_{k=n-q+1}^n \| \bfe_x^k\|_{\bfK^k}^2 + c \, \epsilon_h^n \leq \rho \sum_{k=n-q+1}^n \| \bfe_w^k\|_{\bfK^k}^2 + c \, \epsilon_h^n, \label{eq:Estimate Dr1 1}
\end{align}
where in the last term we used the energy estimate for $\bfe_x$, see \eqref{eq:EnergyBx}.

The term $(IV)$, and similarly $(V)$, is estimated using the factorization $\delta(\zeta) = (1-\zeta) \sigma(\zeta)$ of the generating polynomial $\delta$ of the backwards difference, cf. \eqref{eq:BDF generating function}, see \cite[equation~(10.12)--(10.14)]{KovacsLiLubich2019}, the discrete product rule and \eqref{eq:dG Est inText}, to
\begin{align}
	%\tau \sum_{k=q+1}^n (\bfde_w^k)^\top \big(\bfF(\bftx^k,\bftu^k) - \bfF(\bftx_\ast^k,\bftu_\ast^k)\big) \bftw_\ast^k 
	(IV) =&\ \tau \sum_{k=q+1}^n \sum_{i=0}^{q-1} \sigma_i \partial^\tau \Big((\bfe_w^{k-i})^\top \big(\bfF(\bftx^k,\bftu^k) - \bfF(\bftx_\ast^k,\bftu_\ast^k)\big) \Big)\bftw_\ast^k \nonumber \\
	&\ + \sigma_i (\bfe_w^{k-i-1})^\top \partial^\tau \big(\bfF(\bftx^k,\bftu^k) - \bfF(\bftx_\ast^k,\bftu_\ast^k)\big)\bftw_\ast^k \nonumber \\
	\leq&\ c_2 \sum_{k=n-q+1}^n \| \bfe_u^k \|_{\bfK^k}^2 + \rho \sum_{k=n-q+1}^n \| \bfe_w^k \|_{\bfM^k}^2 \nonumber \\
	&\ + \rho \tau \sum_{k=q+1}^n \| \bfde_u^k \|_{\bfK^k}^2 + c \, \epsilon_h^n, \label{eq:Estimate Dr1 2}
\end{align}
where we used \eqref{eq:dG Est inText} and the regularity of the exact solution to estimate the terms remaining from the telescope sum.

For the term $(VI)$ we use the argument from \cite[equation~(6.38)]{BullerjahnKovacs2024}, to obtain
\begin{align}
	\tau \sum_{k=q+1}^n (\bfde_w^k)^\top \tM_\ast^k \bfd_u^k \leq&\ \rho \sum_{k=n-q+1}^n \| \bfe_w^k \|_{\bfK^k}^2 + c \, \epsilon_h^n + c \, D_h^n. \label{eq:Estimate Dr1 3}
\end{align}

Finally, the most difficult term $(III)$ requires a new argument to eliminate the discrete derivative, since we have a product with error terms in $\bfw$. The semi-discrete proof \cite[equation~(5.40)]{KovacsLiLubich2021} treats this term by a product rule similar to the estimate on the left-hand side. We have already seen that this needs Lemma~\ref{Lemma:Dahlquist} and Lemma~\ref{Lemma:MultiplierTechnique} in the fully discrete setting, but these only allow estimates from below, which is why we need the suitable estimate Lemma~\ref{Lemma:DahlquistUpper} for an upper bound, see Remark~\ref{Rem:UpperBound}. 

Following this plan, using the new Lemma~\ref{Lemma:DahlquistUpper}, we estimate
\begin{align}
	%&\ \tau \sum_{k=q+1}^n (\bfde_w^k)^\top \bfF(\bftx^k,\bftu^k) \bfte_w^k \nonumber \\
	(III)\leq&\ \sum_{k=q+1}^n \sum_{i,j=1}^q B_{ij} (\bfe_w^{k-q+i})^\top \bfF(\bftx^k,\bftu^k) \bfe_w^{k-q+j} - \sum_{i,j=1}^q B_{ij} (\bfe_w^{k-q+i-1})^\top \bfF(\bftx^k,\bftu^k) \bfe_w^{k-q+j-1} \nonumber \\
	=&\ \sum_{i,j=1}^q B_{ij} (\bfe_w^{n-q+i})^\top \bfF(\bftx^n,\bftu^n) \bfe_w^{n-q+j} - \sum_{i,j=1}^q B_{ij} (\bfe_w^{i})^\top \bfF(\bftx^q,\bftu^q) \bfe_w^{j} \nonumber \\
	&\ + \sum_{k=q+1}^n \sum_{i,j=1}^q B_{ij} (\bfe_w^{k-q+i-1})^\top (\bfF(\bftx^{k-1},\bftu^{k-1})-\bfF(\bftx^k,\bftu^k)) \bfe_w^{k-q+j-1} \nonumber \\
	\leq&\ c_3 \sum_{k=n-q+1}^n \| \bfe_w^k\|_{\bfM^k}^2 + \rho \tau \sum_{k=q+1}^n \| \bfde_u^k \|_{\bfK^k}^2 + c \, \epsilon_h^n, \label{eq:Estimate Dr1 4}
\end{align}
where in the last step we have used estimate \eqref{eq:oTau bound F}, as in \eqref{eq:EnergyILHS}, together with \eqref{eq:W1InftyEstimates}.

Inserting the estimates \eqref{eq:Estimate Dr1 1}, \eqref{eq:Estimate Dr1 2}, \eqref{eq:Estimate Dr1 3}, and \eqref{eq:Estimate Dr1 4} into \eqref{eq:Estimate Dr1 tested}, we obtain
\begin{align}
	\tau \sum_{k=q+1}^n (\bfde_w^k)^\top \bfr_1^k \leq&\ c_3 \sum_{k=n-q+1}^n \| \bfe_w^k\|_{\bfM^k}^2 + c_2 \sum_{k=n-q+1}^n \| \bfe_u^k \|_{\bfK^k}^2 + \rho \tau \sum_{k=q+1}^n \| \bfde_u^k \|_{\bfK^k}^2 \nonumber \\
	&\ + \rho \sum_{k=n-q+1}^n \| \bfe_w^k\|_{\bfK^k}^2 + c \, \epsilon_h^n + c \, D_h^n. \label{eq:Estimate Dr1}
\end{align} 

Now combining the estimate \eqref{eq:Estimate Dr1} with \eqref{eq:Energy4 tested}, we finally arrive at the energy estimate (iv):
\begin{align}
	\sum_{k=n-q+1}^n \| \bfe_w^k \|_{\bfA^k}^2 + \tau \sum_{k=q+1}^n \| \bfde_u^k \|_{\bfA^k}^2 \leq&\ c_3 \sum_{k=n-q+1}^n \| \bfe_w^k\|_{\bfM^k}^2 + c_2 \sum_{k=n-q+1}^n \| \bfe_u^k \|_{\bfK^k}^2  \nonumber \\ 
	&\ + c_1 \tau \sum_{k=q+1}^n \|\bfde_u^k\|_{\bfM^k}^2  + \rho \sum_{k=n-q+1}^n \| \bfe_w^k\|_{\bfK^k}^2 \nonumber \\
	&\ + \rho \tau \sum_{k=q+1}^n \|\bfde_u^k\|_{\bfK^k}^2 + c \, \epsilon_h^n + c \, D_h^n. \label{eq:EnergyIV}
\end{align}

(A.3) \emph{Combining the energy estimates }(i)--(iv): We multiply \eqref{eq:EnergyIII} with $2\max\{c_1,c_3\}$ and sum with \eqref{eq:EnergyIV}, absorb terms (by the previously chosen factor and choosing $\rho>0$ sufficiently small), to obtain 
\begin{align}
	\sum_{k=n-q+1}^n \| \bfe_w^k \|_{\bfK^k}^2 + \tau \sum_{k=q+1}^n \| \bfde_u^k \|_{\bfK^k}^2 \leq&\ c_2 \sum_{k=n-q+1}^n \| \bfe_u^k \|_{\bfK^k}^2 + c \, \epsilon_h^n + c \, D_h^n. \label{eq:EnergyA.2}
\end{align}

Finally, we again multiply \eqref{eq:EnergyA.1} with $2c_2$ and sum with \eqref{eq:EnergyA.2}, and absorb terms, to arrive at the final energy estimate of Part A:
\begin{align}
	\sum_{k=n-q+1}^n \| \bfe_u^k \|_{\bfK^k}^2 + \tau \sum_{k=q+1}^n \| \bfe_w^k \|_{\bfK^k}^2 + \sum_{k=n-q+1}^n \| \bfe_w^k \|_{\bfK^k}^2 + \tau \sum_{k=q+1}^n \| \bfde_u^k \|_{\bfK^k}^2 \leq&\ c \, \epsilon_h^n + c \, D_h^n. \label{eq:EnergyA}
\end{align}

\emph{(B) Estimates for the velocity equation:} The estimate for the error in velocity,
\begin{align}
	\| \bfe_v^n \|_{\bfK(\bfx_\ast^n)} \leq c \, \big(\|\bfe_u^n \|_{\bfK^n} + \| \bfe_w^n \|_{\bfK^n}\big) + \| \bfd_v \|_{\bfK(\bfx_\ast^n)}, \label{eq:EnergyBv}
\end{align}
is obtained exactly as in the semi-discrete proof \cite[equation~(5.44)]{KovacsLiLubich2021}, using \eqref{eq:error eq - v}, \eqref{eq:W1InftyEstimates} and \cite[Lemma~5.3]{KovacsLiLubich2021}. 

%\begin{lem}[{\cite[Lemma~5.3]{KovacsLiLubich2021}}]
%\label{lemma:interpolation} 
%	For an admissible triangulation of a smooth surface  $\Gamma$, let $\Gamma_h^*$ be the interpolated surface with finite elements of polynomial degree $k\ge 1$.
%	Let  $\widetilde I_h^*:C(\Gamma_h^*)\to S_h(\Gamma_h^*)$ denote the finite element interpolation operator on $\Ga_h^*$. Then, the interpolation of the product of two finite element functions $a_h, b_h$ on $\Gamma_h^*$ is bounded by
%	$$
%	\| \widetilde I_h^* (a_h b_h) \|_{H^1(\Gamma_h^*)}  \le C\, \| a_h \|_{H^1(\Gamma_h^*)} \, \| b_h \|_{W^{1,\infty}(\Gamma_h^*)},
%	$$
%	where $C$ depends only on $\Gamma$ (more precisely, on bounds of higher derivatives of a parametrization of $\Gamma$), on shape-regularity and quasi-uniformity of the triangulation, and on the degree $k$.
%\end{lem}

The error estimate for the position is obtained, as in \cite[equation~(10.36)]{KovacsLiLubich2019}, by the estimate 
\begin{align*}
	\| \bfe_x^n \|_{\bfK^n}^2 \leq c \, \tau \sum_{k=q}^n \|\bfde_x^n \|_{\bfK^n}^2 + c \sum_{i=0}^{q-1} \| \bfe_x^i \|_{\bfK^n}^2.
\end{align*}
Together with \eqref{eq:error eq - x} herein, this yields
\begin{align}
	\| \bfe_x^n \|_{\bfK^n}^2 \leq c \, \tau \sum_{k=q}^n \| \bfe_v^k \|_{\bfK^k}^2 + c \, I_h^{q-1} + c \, D_h^n. \label{eq:EnergyBx}
\end{align}

\emph{(C) Combination:} Combining the bounds \eqref{eq:EnergyA}, \eqref{eq:EnergyBv}, and \eqref{eq:EnergyBx}, using a discrete Gr\"onwall inequality and the equivalence between the norms corresponding to $\bftx^n$ and $\bfx_\ast^n$ yields the stability estimate \eqref{eq:StabilityBound}. 

Finally, we also obtain $t_{\text{max}}=T$ by the assumed bounds on the defects \eqref{eq:ConditionDefectsSatbility} and initial data \eqref{eq:AssInitialValues} and the step size restriction $\tau^q \leq C_0 h^k$, as in \cite{KovacsLiLubich2019}.

\end{proof}

\begin{rem}
	Note that we only need to restrict the order of the BDF method to $1\leq q \leq2$ in order to obtain the estimate \eqref{eq:Estimate Dr1 4}, via Lemma~\ref{Lemma:DahlquistUpper}. All other estimates can be shown for $1 \leq q \leq 5$ by an adaptation of the presented arguments involving the multiplier $\eta_q$ from Lemma~\ref{Lemma:MultiplierTechnique}, see, e.g., the stability proof in \cite[Section~6.4]{BullerjahnKovacs2024}.
	
	By these necessary adaptations, the resulting critical term in \eqref{eq:Estimate Dr1 4} takes the form
	\begin{equation*}
		\tau \sum_{k=q+1}^n (\bfde_w^k)^\top \bfF(\bftx^k,\bftu^k) (\bfte_w^k - \eta_q \bfte_w^{k-1}),
	\end{equation*}
	and as already indicated in Remark~\ref{rem: Upper bound q345}, for which -- up to our knowledge -- a similar upper bound is not yet available. However, we hope that this can be overcome in future work and, for this reason, we have kept most of our estimates generally applicable for the $q$-step BDF methods for $q=3,4,5$.

\end{rem}

\section{Consistency and Proof of Theorem~\ref{Theorem:MainConvergence}} \label{Sec:Consistency}
In this section we show that the fully discrete defects and their discrete derivatives are bounded by $\calO(h^k+\tau^q)$ in the appropriate norms appearing in the stability bounds of Proposition~\ref{Prop:StabilityEst}. Together with the bounds in the initial values, see \cite[Proposition~6.2]{KovacsLiLubich2021}, this completes the proof of Theorem~\ref{Theorem:MainConvergence} by a standard argument combining the stability and consistency arguments, see, e.g., \cite[Section~9]{KovacsLiLubich2019}.

The fully discrete defect terms are estimated by splitting the defect terms in a spatial and temporal part, the spatial part is estimated by the semi-discrete defect bounds in \cite[Section~6.1]{KovacsLiLubich2021}, while the temporal part is estimated following \cite[Section~11]{KovacsLiLubich2019}. The most difficult part here are the discrete derivative of the defect terms, where we apply the techniques developed in \cite[Section~7]{BullerjahnKovacs2024}. A careful combination of these arguments for the various terms yields the following result.

\begin{prop}
\label{Prop:Consistency}
	Under the assumptions of Theorem~\ref{Theorem:MainConvergence}, the defects $d_x^n \in S_h(\Ga_h[x_\ast^n])^3$, $d_v^n \in S_h(\Ga_h[x_\ast^n])^3$, $d_u^n \in S_h(\Ga_h[x_\ast^n])^4$, and $d_w^n \in S_h(\Ga_h[x_\ast^n])^4$ of the $k$th-degree finite elements and $q$-step backward difference formula, as defined by their nodal vectors $\bfd_x$, $\bfd_v$, $\bfd_u$, and $\bfd_w$, respectively in \eqref{eq:matrix-form-X-v-star} and \eqref{eq:defect vectors}, are bounded by
	\begin{equation}
		\label{eq:defect bounds}
		\begin{alignedat}{4}
			&\| \bfd_x^n \|_{\bfK(\bfx_\ast^n)} \leq c \, \tau^q,  \qquad &&\| \bfd_v^n \|_{\bfK(\bfx_\ast^n)} \leq c h^k , \\ 
			&\| \bfd_u^n \|_{\ast,\bfx_\ast^n} \leq c(h^k + \tau^q), \qquad &&\| \bfdd_u^n \|_{\ast,\bfx_\ast^n} \leq c(h^k + \tau^q), \\
			&\| \bfd_w^n \|_{\ast,\bfx_\ast^n} \leq c(h^k + \tau^q), \qquad &&\| \bfdd_w^n \|_{\ast,\bfx_\ast^n} \leq c(h^k + \tau^q).
		\end{alignedat}
	\end{equation}
	The constant $c>0$ is independent of $h$, $\tau$ and $n$ with $n\tau \leq T$.
\end{prop}
\begin{proof}
	The bound in $\bfd_x^n$ is obtained exactly as in \cite[Lemma~11.1]{KovacsLiLubich2019}. The term $\bfd_v^n$ does not contain any temporal approximation, therefore the bound simply follows by the semi-discrete bound in \cite[Proposition~6.1]{KovacsLiLubich2021}. The term $\bfd_u^n$ is split, as in \cite[Lemma~11.1]{KovacsLiLubich2019}, into
	\begin{align}
		\bfM(\bftx_\ast^n) \bfd_u^n =&\ \bfM(\bfx_\ast^n) \bfd_u(t_n) + \bfM(\bftx_\ast^n) (\bfdu_\ast^n - \partial_t \bfu_\ast(t_n)) + (\bfM(\bftx_\ast^n)-\bfM(\bfx_\ast^n)) \partial_t \bfu_\ast(t_n) \nonumber \\
		&\ + (\bfA(\bftx_\ast^n) - \bfA(\bfx_\ast^n)) \bfw_\ast^n - \bfF(\bftx_\ast^n,\bftu_\ast^n) (\bftw_\ast^n - \bfw_\ast(t_n)) \nonumber \\
		&\ - (\bfF(\bftx_\ast^n,\bftu_\ast^n)-\bfF(\bfx_\ast^n,\bfu_\ast^n)) \bfw_\ast(t_n) - (\bff(\bftx_\ast^n,\bftu_\ast^n) - \bff(\bfx_\ast^n,\bfu_\ast^n)),
	\end{align}
	and then estimated by either the spatial defect estimates shown in \cite[Proposition~6.1]{KovacsLiLubich2021}, or by the temporal estimates following \cite[Lemma~11.1]{KovacsLiLubich2019}.
	
	The estimates for the discrete derivative of the defect terms is split in the same way and estimated by the techniques developed in \cite[Section~7.3]{BullerjahnKovacs2024}, extending the notion of the backward difference and extrapolation by a Hermite interpolation an estimating via a Peano kernel representation. Together with the estimates \eqref{eq:oTau bound} and splitting the non-linear part, e.g., as in \eqref{eq:oTau bound F}, this yields the stated bounds \eqref{eq:defect bounds}.
\end{proof}

\medskip
As in \cite{KovacsLiLubich2021} for the semi-discrete case, the fully discrete error estimates follow by combining the stability estimates, Proposition~\ref{Prop:StabilityEst}, the consistency bounds, Proposition~\ref{Prop:Consistency}, and the error estimates for the Ritz map and the interpolation operator.

\section{Numerical experiments} \label{Sec:Numerical Results}

In order to illustrate and complement our theoretical findings in Theorem~\ref{Theorem:MainConvergence}, we present two numerical experiments:
\begin{enumerate}
	\item A spatial convergence test using stationary solutions of Willmore flow, i.e., a sphere and a Clifford torus, repeating the experiments of \cite[Section~7.1]{KovacsLiLubich2021} using the modified system.
	\item A temporal convergence test using a reference solution of Willmore flow on an ellipsoid or oblate spheroid with semi-axis lengths of $2,2$ and $1$.
\end{enumerate}

The numerical experiments were performed using quadratic evolving surface finite elements, implemented using $\ell$FEM \cite{ellFEM}, and linearly implicit backward difference method of order $1$ and $2$, in \textsc{Matlab}. The initial meshes for all surfaces were generated using DistMesh \cite{PerssonStrang2004}, without taking advantage of the symmetries of the surfaces. As initial values for the BDF method, we used the nodal interpolation of the exact initial values. In all Figures the $L^\infty(H^1)$-norm errors for the surface $X$, normal vector $\nu$ and mean curvature $H$ are shown against the time step size, and respectively mesh size, in a logarithmic plot. 

%\begin{figure}[!t]
%	\centering\includegraphics[width=\textwidth]{Figures/convplot_Willmore_Sphere_T2_BDF2_space_Linfty}
%	\caption{Temporal convergence plot for the BDF2/quadratic ESFEM approximation of the Willmore flow for the sphere.}
%	\label{fig:TempConvSphere}
%\end{figure}

\begin{figure}[hbp]
	\centering\includegraphics[width=\textwidth]{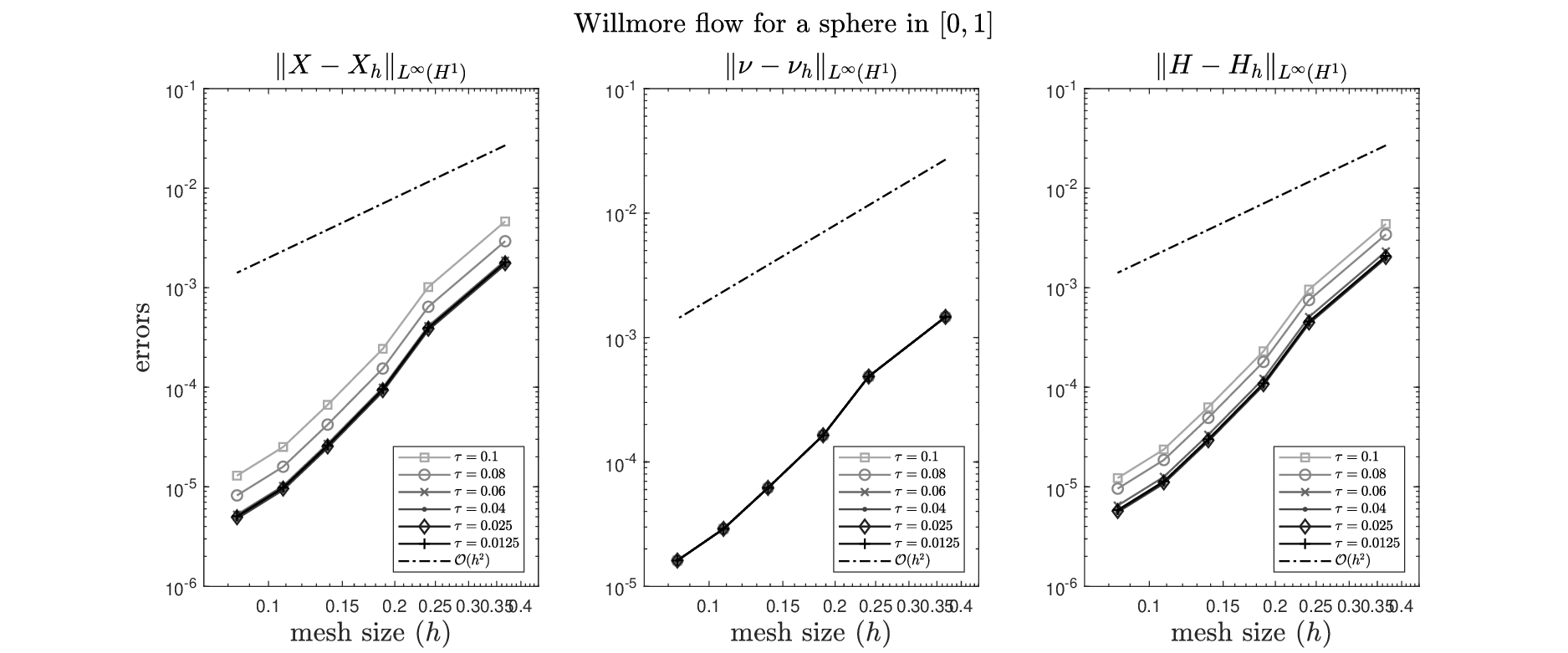}
	\caption{Spatial convergence plot for the BDF2/quadratic ESFEM approximation of the Willmore flow for the sphere.}
	\label{fig:SpaceConvSphere}
\end{figure}

In Figure~\ref{fig:SpaceConvSphere}, we report on the errors between the numerical and exact solutions for the Willmore flow of a sphere of radius $R=1$ on the time interval $[0,1]$, using the BDF2 method. We used a sequence of time step sizes $\tau_k=0.1,0.08,0.06,0.04,0.025,0.0125$ and a sequence of meshes with mesh sizes $h_j=0.2405,0.18724,0.13828,0.1083,0.084061,0.075318$.

\begin{figure}[htbp]
	\centering\includegraphics[width=\textwidth]{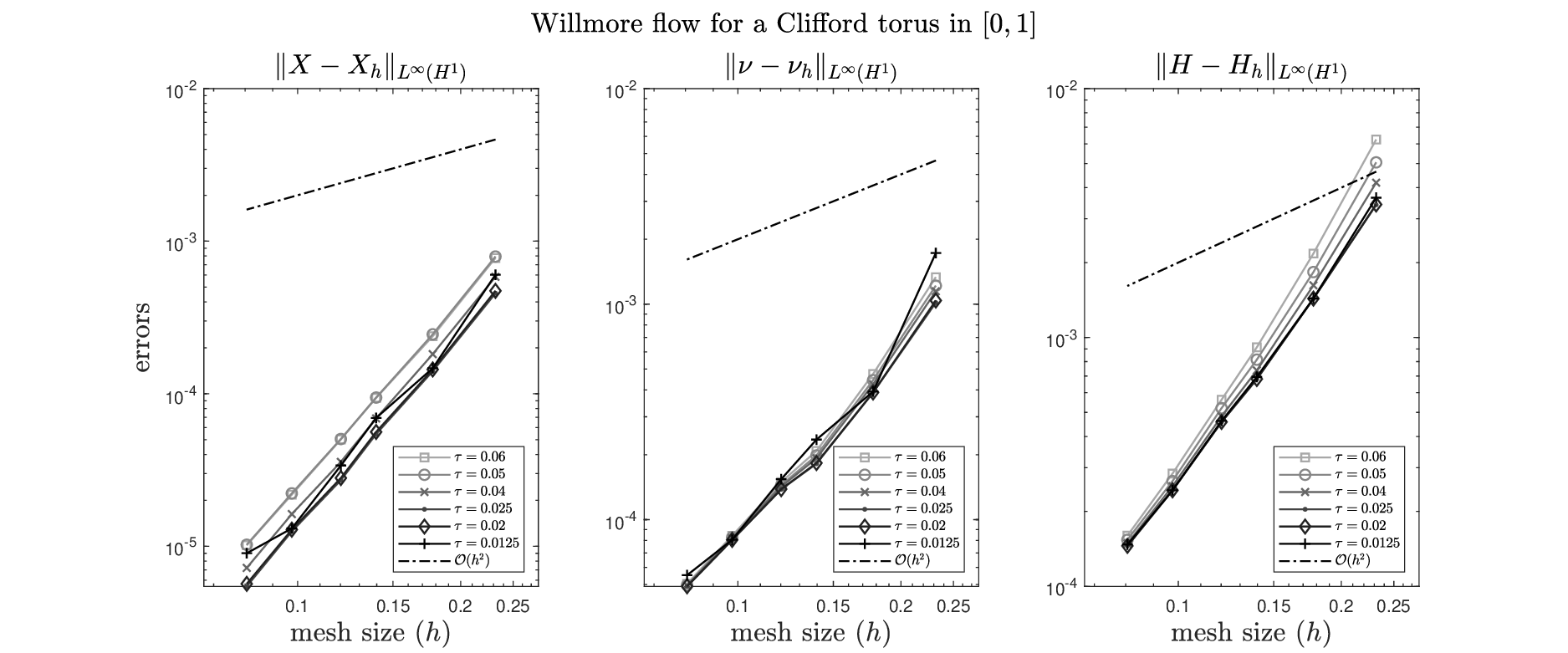}
	\caption{Spatial convergence plot for the BDF1/quadratic ESFEM approximation of the Willmore flow for the Clifford torus.}
	\label{fig:SpaceConvTorus}
\end{figure}

While in Figure~\ref{fig:SpaceConvTorus}, we report on the errors between the numerical and exact solutions for the Willmore flow of a Clifford torus with major radius $R=1$ and minor radius $r=1/\sqrt{2}$ on the time interval $[0,1]$, using graded meshes, refining the area with higher curvature, and the BDF1 method. We used a sequence of time step sizes $\tau_k=0.06,0.05,0.04,0.025,0.02,0.0125$ and a sequence of meshes with mesh sizes $h_j=0.23189,0.17761,0.13964,0.1202,0.097577,0.080559$.

\begin{figure}[htbp]
	\centering\includegraphics[width=\textwidth]{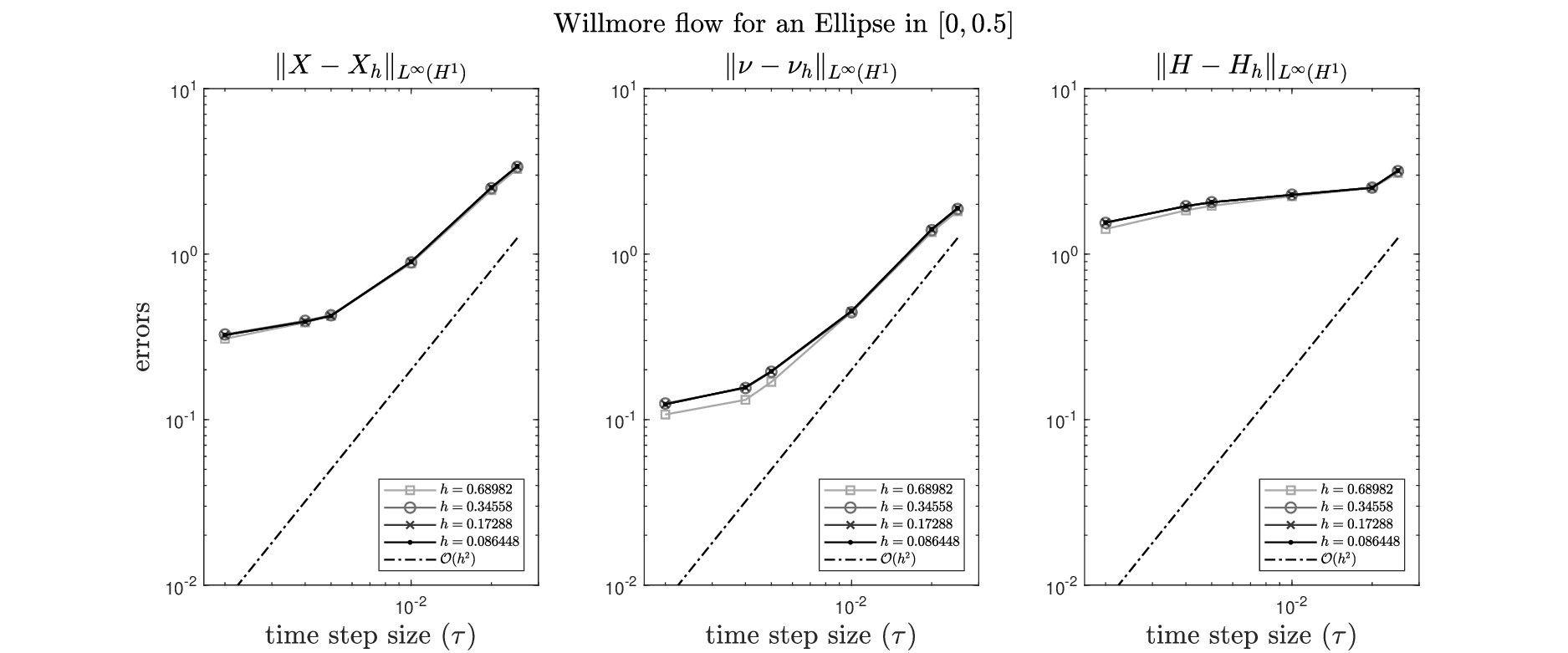}
	\caption{Temporal convergence plot for the BDF2/quadratic ESFEM approximation of the Willmore flow for the ellipsoid.}
	\label{fig:TempConvEll}
\end{figure} 

Finally in Figure~\ref{fig:TempConvEll}, we report on the errors between the numerical and a reference solutions (computed on a refined mesh of size $h=0.043226$ with time step size $\tau=10^{-4}$) for Willmore flow of an ellipsoid with semi-axis lengths of $2.2$ and $1$, on the time interval $[0,0.5]$, using the BDF2 method. We used a sequence of time step sizes $\tau_k=0.025,0.02, 0.01, 0.005, 0.004, 0.002$ and a sequence of meshes with mesh sizes $h_j=0.68982,0.34558,0.17288,0.086448$. 

%\begin{figure}[!t]
%	\centering\includegraphics[width=\textwidth]{Figures/Time_Conv_Ell2}
%	\caption{Temporal convergence plot for the BDF2/quadratic ESFEM approximation of the Willmore flow for the ellipsoid.}
%	\label{fig:TempConvEll2}
%\end{figure} 
%
%Finally in Figure~\ref{fig:TempConvEll2}, we report on the errors between the numerical and a reference solutions (computed on a refined mesh of size $h=0.043226$ with time step size $\tau=10^{-4}$) for Willmore flow of an ellipsoid or oblate spheroid with semi-axis lengths of $2,2$ and $1$, on the time interval $[0,0.5]$, using the BDF2 method. We used a sequence of time step sizes $\tau_k=0.125, 0.025,0.02, 0.01, 0.005, 0.004, 0.002$ and a sequence of meshes with mesh sizes $h_j=0.68982,0.34558,0.17288,0.086448$. 

In the first experiment the observed convergence in space as shown in Figures~\ref{fig:SpaceConvSphere} and \ref{fig:SpaceConvTorus}, is in agreement with the theoretical results of Theorem~\ref{Theorem:MainConvergence} (note the reference lines) and replicates the experiments already conducted in \cite{KovacsLiLubich2021}. Since the temporal evolution is stationary in the first example we conducted the second experiment, where the observed convergence in time as shown in Figures~\ref{fig:TempConvEll} again illustrates the theoretical results.

\bibliography{willmore}% common bib file
%\input sn-article.bbl

\end{document}